\documentclass{article}%
\usepackage{amsfonts}
\usepackage{amsmath}
\usepackage{geometry}
\usepackage{amssymb}
\usepackage{graphicx}
\usepackage{url}
\usepackage{hyperref}
\usepackage{accents}
\usepackage{afterpage}
\usepackage{endnotes}
\usepackage{lastpage}%
\providecommand{\U}[1]{\protect\rule{.1in}{.1in}}
\newtheorem{theorem}{Theorem}[section]
\newtheorem{lemma}[theorem]{Lemma}

\numberwithin{equation}{section}

\newtheorem{proposition}[theorem]{Proposition}

\newenvironment{proof}[1][Proof]{\noindent\textbf{#1.} }{\ \rule{0.5em}{0.5em}}
\begin{document}

\title{Rank $3$ filtered $(\varphi,N)$-modules}
\author{Gerasimos Dousmanis}
\maketitle

\begin{abstract}
We classify $3$-dimensional semistable Frob-semisimple representations of $G_{%
%TCIMACRO{\U{211a} }%
%BeginExpansion
\mathbb{Q}
%EndExpansion
_{p^{f}}}.$

\end{abstract}

In this paper we classify $3$-dimensional semistable Frob-semisimple
representations of $G_{K}$ where $K$ is the unramified extension of
$\mathbb{Q}_{p}$ of degree $f.$ As in \cite{1}, this easily extends to the
case of $3$-dimensional semisimple representations of $G_{K}$ for any finite
extension $K$ of $\mathbb{Q}_{p},\ $including the non-Frob-semisimple and
Frob-scalar cases. The paper was written in June $2012$ during the Thematic
Program on Galois Representations at Fields Institute. We thank the Fields
Institute for its hospitality and financial support.

Let $p$\ will be a fixed odd integer prime, $K_{f}=\mathbb{Q}_{p^{f}}$\ the
finite unramified extension of $\mathbb{Q}_{p}$\ of degree $f,$\ and $E$\ a
finite, large enough extension of $K_{f}.$\ When the degree of $K_{f}$\ plays
no role we simply write $K.$\ We denote by $\sigma_{K}$\ the absolute
Frobenius of $K;$\ we fix once and for all a distinguished embedding
$K\overset{\tau_{0}}{\hookrightarrow}E$\ and we let $\tau_{i}=\tau_{0}%
\circ\sigma_{K}^{i}$\ for all $i=0,1,...,f-1.$\ We fix the $f$-tuple of
embeddings $\mid\tau\mid:=(\tau_{0},\tau_{1},...,\tau_{f-1})\ $and we denote
$E^{\mid\mathcal{\tau}\mid}:=%
%TCIMACRO{\tprod \nolimits_{\tau:K\hookrightarrow E}}%
%BeginExpansion
{\textstyle\prod\nolimits_{\tau:K\hookrightarrow E}}
%EndExpansion
E,\ $with the embeddings ordered as above. The map $\nu:K\otimes_{%
%TCIMACRO{\U{211a} }%
%BeginExpansion
\mathbb{Q}
%EndExpansion
_{p}}E\rightarrow E^{\mid\mathcal{\tau}\mid}\ $with$\ \nu(x\otimes
y):=(\tau_{i}(x)y)_{\tau_{i}}$ is a ring isomorphism and the ring automorphism
$1_{E}\otimes\sigma:E\otimes K\rightarrow E\otimes K$ transforms via $\nu$ to
the automorphism $\mathcal{\varphi}:E^{\mid\mathcal{\tau}\mid}\rightarrow
E^{\mid\mathcal{\tau}\mid}$ with $\mathcal{\varphi}(x_{0},x_{1},...,x_{f-1}%
)=(x_{1},...,x_{f-1},x_{0}).$ Notice that $\varphi^{f}$ acts trivially on
$E^{\mid\mathcal{\tau}\mid}.$ We denote by $e_{\tau_{i}}=(0,...,1_{i},...,0)$
the idempotent of $E^{\mid\mathcal{\tau}\mid}$ where the $1$ occurs in the
$\tau_{i}$-th coordinate, for each $i\in\{0,1,...,f-1\}.$ \smallskip For any
$\vec{x}\in E^{\mid\mathcal{\tau}\mid},$ we denote by $x(i)$ its $i$-th
coordinate, by $\mathrm{Nm}_{\varphi}(\vec{x})\ $the vector $%
%TCIMACRO{\tprod \nolimits_{i=0}^{f-1}}%
%BeginExpansion
{\textstyle\prod\nolimits_{i=0}^{f-1}}
%EndExpansion
\varphi^{i}(\vec{x});\ $we write $\mathrm{Nm}_{\varphi}(\vec{x})$ for the
scalar $%
%TCIMACRO{\tprod \nolimits_{i=0}^{f-1}}%
%BeginExpansion
{\textstyle\prod\nolimits_{i=0}^{f-1}}
%EndExpansion
x(i)$ such that $\mathrm{Nm}_{\varphi}(\vec{x})=%
%TCIMACRO{\tprod \nolimits_{i=0}^{f-1}}%
%BeginExpansion
{\textstyle\prod\nolimits_{i=0}^{f-1}}
%EndExpansion
x(i)\cdot\vec{1},$ and we define $v_{p}\left(  \mathrm{Nm}_{\varphi}(\vec
{x})\right)  :=$ $v_{p}(%
%TCIMACRO{\tprod \nolimits_{i=0}^{f-1}}%
%BeginExpansion
{\textstyle\prod\nolimits_{i=0}^{f-1}}
%EndExpansion
x(i)).$ For any matrix $A\ $with entries in $E^{\mid\mathcal{\tau}\mid}$\ we
write $\mathrm{Nm}_{\varphi}(A):=A\varphi(A)\cdot\cdot\cdot\varphi^{f-1}%
(A),$\ with $\varphi$\ acting on each entry of $A.$

\section{Rank three weakly admissible filtered $\varphi$-modules over $K$ with
$E$-coefficients and distinct eigenvalues of frobenius}

Let $\mathrm{MF}_{K,E}^{\varphi}$ be the category of filtered $\varphi
$-modules over $K$ with $E$-coefficients, and let $\mathrm{MF}_{K,E}%
^{\varphi,\mathrm{w.a.}}$ be the subcategory of weakly admissible filtered
$\varphi$-submodules. Let $\left(  D,\varphi\right)  $ be a three-dimensional
object in $\mathrm{MF}_{K,E}^{\varphi}.$ We let $D_{i}:=e_{\tau_{i}}D_{i}$ and
we define $\mathrm{Fil}^{j}D_{i}:=e_{\tau_{i}}\mathrm{Fil}^{j}D$ for all $i$
and $j,$ and for each $i=0,1,...,f-1.$ We define the labeled Hodge-Tate
weights $\mathrm{HT}\left(  D_{i}\right)  \ $of $D\ $with respect to the
embedding $\tau_{i}$ to be the multiset consisting of the opposites of the
jumps of the filtration of $D_{i},$ and after twisting by some rank one weakly
admissible filtered $\varphi$-module we may assume that $\mathrm{HT}\left(
D_{i}\right)  =\{0=k_{0}(i)\leq k_{1}(i)\leq k_{2}(i)\}$ for all $i.$ We say
that the eigenvalues of frobenius of $\left(  D,\varphi\right)  $ are distinct
if $\mathrm{Mat}_{\underline{e}}\left(  \varphi\right)  =\mathrm{diag}\left(
\vec{a},\vec{b},\vec{c}\right)  $ for some ordered basis $\underline{e}\ $and
some $\vec{a},\vec{b},\vec{c}\in\left(  E^{\times}\right)  ^{f}$ with distinct
$\mathrm{Nm}_{\varphi}\left(  \vec{a}\right)  ,\mathrm{Nm}_{\varphi}(\vec
{b})\ $and $\mathrm{Nm}_{\varphi}\left(  \vec{c}\right)  .$ This is equivalent
to $\mathrm{Mat}_{\underline{e}}\left(  \varphi^{f}\right)  $ being diagonal
with distinct diagonal entries of the form $\alpha\cdot\vec{1},$ where
$\alpha\in E^{\times}$ (see Proposition \ref{dist eigen}). We classify rank
$3$ objects of $\mathrm{MF}_{K,E}^{\varphi}$ up to isomorphism and give a
criterion for weak admissibility.

\begin{proposition}
\label{phi, fil}Any rank $3$ filtered $\varphi$-module over $K$ with
$E$-coefficients, distinct eigenvalues of frobenius, and labeled Hodge-Tate
weights $\mathrm{HT}\left(  D_{i}\right)  =\{0=k_{0}(i)\leq k_{1}(i)\leq
k_{2}(i)\}$ for all $i$ is isomorphic to a filtered $\varphi$-module $\left(
D,\varphi\right)  $ with frobenius action defined by a matrix of the form%
\[
\mathrm{Mat}_{\underline{e}}\left(  \varphi\right)  =\mathrm{diag}\left(
\vec{a},\vec{b},\vec{c}\right)
\]
for some ordered basis $\underline{e}$ of $D$ and some $\vec{a},\vec{b}%
,\vec{c}\in\left(  E^{\times}\right)  ^{f}$ with distinct norms $\mathrm{Nm}%
_{\varphi}\left(  \vec{a}\right)  ,\mathrm{Nm}_{\varphi}(\vec{b})\ $and
$\mathrm{Nm}_{\varphi}\left(  \vec{c}\right)  ,$ and filtration%
\[
\mathrm{Fil}^{\mathit{j}}\left(  D\right)  =%
%TCIMACRO{\tbigoplus \limits_{i=0}^{f-1}}%
%BeginExpansion
{\textstyle\bigoplus\limits_{i=0}^{f-1}}
%EndExpansion
\mathrm{Fil}^{\mathit{j}}\left(  D_{i}\right)  ,
\]
where:

\begin{enumerate}
\item If $\mathrm{HT}\left(  D_{i}\right)  =\{0=k_{0}(i)<k_{1}(i)<k_{2}(i)\},$
then%
\begin{equation}
\ \ \text{\ \ \ \ }\mathrm{Fil}^{j}D_{i}=\left\{
\begin{array}
[c]{l}%
e_{\tau_{i}}E^{\mid\mathcal{\tau}\mid}e_{0}\oplus e_{\tau_{i}}E^{\mid
\mathcal{\tau}\mid}e_{1}\oplus e_{\tau_{i}}E^{\mid\mathcal{\tau}\mid}%
e_{2}\ \ \ \ \ \ \ \ \ \ \ \ \ \ \ \ \ \mathrm{if}\text{ }j\leq0,\\
e_{\tau_{i}}E^{\mid\mathcal{\tau}\mid}\left(  e_{0}+x_{2}(i)e_{2}\right)
\oplus e_{\tau_{i}}E^{\mid\mathcal{\tau}\mid}\left(  e_{1}+x_{2}^{\prime
}(i)e_{2}\right)  \ \ \mathrm{if}\text{ }1\leq j\leq k_{1}(i),\\
e_{\tau_{i}}E^{\mid\mathcal{\tau}\mid}\left(  e_{0}+x_{1}(i)e_{1}%
+x_{2}^{\prime\prime}(i)e_{2}\right)
\ \ \ \ \ \ \ \ \ \ \ \ \ \ \ \ \ \ \ \mathrm{if}\text{ }1+k_{1}(i)\leq j\leq
k_{2}(i),\\
\ \ \ \ \ \ \ \ \ \ \ \ \ \ 0\ \ \ \ \ \ \ \ \ \ \ \ \ \ \ \ \ \ \ \ \ \ \ \ \ \ \ \ \ \ \ \ \ \ \ \ \ \ \ \ \ \ \mathrm{if}%
\text{ }1+k_{2}(i)\leq j,
\end{array}
\right.  \tag{$\mathcal{F}_0$}\label{1}%
\end{equation}
with $x_{1}(i)\in E,\ x_{2}(i),x_{2}^{\prime}(i)\in\{0,1\},$ and
$x_{2}^{\prime\prime}(i)=x_{2}(i)+x_{1}(i)x_{2}^{\prime}(i).$

\item If $\mathrm{HT}\left(  D_{i}\right)  =\{0=k_{0}(i)<k_{1}(i)=k_{2}%
(i)\}\ $or$\ \{0=k_{0}(i)=k_{1}(i)<k_{2}(i)\},$ then
\begin{equation}
\ \ \ \ \ \ \ \ \ \ \ \ \ \ \ \ \ \ \mathrm{Fil}^{j}D_{i}=\left\{
\begin{array}
[c]{l}%
e_{\tau_{i}}E^{\mid\mathcal{\tau}\mid}e_{0}\oplus e_{\tau_{i}}E^{\mid
\mathcal{\tau}\mid}e_{1}\oplus e_{\tau_{i}}E^{\mid\mathcal{\tau}\mid}%
e_{2}\ \ \ \ \ \ \ \ \ \ \ \ \ \ \ \ \ \ \ \mathrm{if}\text{ }j\leq0,\\
e_{\tau_{i}}E^{\mid\mathcal{\tau}\mid}\left(  e_{0}+x_{2}(i)e_{2}\right)
\oplus e_{\tau_{i}}E^{\mid\mathcal{\tau}\mid}\left(  e_{1}+x_{2}^{\prime
}(i)e_{2}\right)  \ \ \ \ \mathrm{if}\text{ }1\leq j\leq k(i),\\
\ \ \ \ \ \ \ \ \ \ \ \ \ \ \ \ \ \ \ \ \ \ \ \ \ \ 0\ \ \ \ \ \ \ \ \ \ \ \ \ \ \ \ \ \ \ \ \ \ \ \ \ \ \ \ \ \ \ \ \mathrm{if}%
\text{ }j\geq1+k(i),
\end{array}
\right.  \tag{$\mathcal{F}_1$}\label{fil1}%
\end{equation}
with $x_{2}(i),x_{2}^{\prime}(i)\in\{0,1\}\ $and with $k(i)\ $denoting the
nonzero weight, or of the form
\begin{equation}
\mathrm{Fil}^{\mathit{j}}\left(  e_{\tau_{i}}D\right)  =\left\{
\begin{array}
[c]{l}%
e_{\tau_{i}}E^{\mid\mathcal{\tau}\mid}e_{0}\oplus e_{\tau_{i}}E^{\mid
\mathcal{\tau}\mid}e_{1}\oplus e_{\tau_{i}}E^{\mid\mathcal{\tau}\mid}%
e_{2}\ \mathrm{if}\text{ }j\leq0,\\
e_{\tau_{i}}E^{\mid\mathcal{\tau}\mid}\left(  e_{0}+x_{1}(i)e_{1}%
+x_{2}^{\prime\prime}(i)e_{2}\right)  \ \ \mathrm{if}\text{ }1\leq j\leq
k(i),\\
\ \ \ \ \ \ \ \ \ \ \ \ \ \ \ \ \ 0\ \ \ \ \ \ \ \ \ \ \ \ \ \ \ \ \ \ \ \ \ \ \mathrm{if}%
\text{ }j\geq1+k(i),
\end{array}
\right.  \tag{ $\mathcal{F}_2$}\label{filtr'}%
\end{equation}
with $x_{1}(i),x_{2}^{\prime\prime}(i)\in\{0,1\}.$

\item Finally, if $\mathrm{HT}\left(  D_{i}\right)  =\{0\},$ then
\begin{equation}
\mathrm{Fil}^{j}D_{i}=\left\{
\begin{array}
[c]{l}%
e_{\tau_{i}}E^{\mid\mathcal{\tau}\mid}e_{0}\oplus e_{\tau_{i}}E^{\mid
\mathcal{\tau}\mid}e_{1}\oplus e_{\tau_{i}}E^{\mid\mathcal{\tau}\mid}%
e_{2}\ \mathrm{if}\text{ }j\leq0,\\
\ \ \ \ \ \ \ \ \ \ \ \ \ \ \ \ \ \ \ 0\ \ \ \ \ \ \ \ \ \ \ \ \ \ \ \ \ \ \ \ \mathrm{if}%
\text{ }j\geq1.
\end{array}
\right.  \tag{$\mathcal{F}_3$}\label{triv fil}%
\end{equation}
\noindent\bigskip
\end{enumerate}
\end{proposition}

\noindent Let%
\[
\ \ I_{0}:=\{0,1,...,f-1\},
\]%
\[
I_{1}:=\{i\in I_{0}:\mathrm{HT}\left(  D_{i}\right)  =\{0=k_{0}(i)<k_{1}%
(i)<k_{2}(i)\}\},\
\]
and\noindent%
\begin{align*}
&  \ \ \ I_{2}:=\{i\in I_{0}:\mathrm{HT}\left(  D_{i}\right)  =\{0=k_{0}%
(i)=k_{1}(i)<k_{2}(i)\}\ \text{or\ }\mathrm{HT}\left(  D_{i}\right)
=\{0=k_{0}(i)<k_{1}(i)=k_{2}(i)\},\\
&  \text{and\ the\ filtration\ of\ }D_{1}%
\text{\ is\ given\ by\ formula\ (\ref{fil1})}\}.\ \text{In\ this case the
nonzero weight is denoted by}\ k(i).
\end{align*}
Finally, let\noindent%
\begin{align*}
&  I_{3}:=\{i\in I_{0}:\mathrm{HT}\left(  D_{i}\right)  =\{0=k_{0}%
(i)=k_{1}(i)<k_{2}(i)\}\ \text{or\ }\mathrm{HT}\left(  D_{i}\right)
=\{0=k_{0}(i)<k_{1}(i)=k_{2}(i)\},\\
&  \text{and\ the\ filtration\ of\ }D_{1}%
\text{\ is\ given\ by\ formula\ (\ref{filtr'})}\}.
\end{align*}

\begin{proposition}
\label{wa}A filtered $\varphi$-module $\left(  D,\varphi\right)  $ as in
Proposition $\ref{phi, fil}$ is weakly admissible if and only if%
\begin{align}
&  v_{p}\left(  \mathrm{Nm}_{\varphi}\left(  \vec{a}\right)  \right)
+v_{p}\left(  \mathrm{Nm}_{\varphi}(\vec{b})\right)  +v_{p}\left(
\mathrm{Nm}_{\varphi}(\vec{c}\right)  =\sum\limits_{i\in I_{1}}\left(
k_{1}(i)+k_{2}(i)\right)  +\sum\limits_{i\in I_{2}}2k(i)+\sum\limits_{i\in
I_{3}}k(i),\label{9}\\
& \nonumber\\
&  v_{p}\left(  \mathrm{Nm}_{\varphi}\left(  \vec{a}\right)  \right)  \geq
\sum\limits_{\substack{i\in I_{1}\ \mathrm{such\ that}\\x_{2}%
(i)=0\ \mathrm{and}\text{ }x_{1}(i)\not =0}}k_{1}(i)+\sum
\limits_{\substack{i\in I_{1}\ \mathrm{such\ that}\\x_{1}(i)=x_{2}(i)=0}%
}k_{2}(i)\ +\label{10}\\
& \nonumber\\
&  +\sum\limits_{\substack{i\in I_{2}\ \mathrm{such\ that}\\x_{1}%
(i)=0}}k(i)+\sum\limits_{\substack{i\in I_{3}\ \mathrm{such\ that}%
\\x_{1}(i)=x_{2}^{\prime\prime}(i)=0}}k(i),\nonumber
\end{align}%
\begin{equation}
v_{p}\left(  \mathrm{Nm}_{\varphi}(\vec{b})\right)  \geq\sum
\limits_{\substack{i\in I_{1}\ \mathrm{such\ that}\\x_{2}^{\prime}(i)=0}%
}k_{1}(i)+\sum\limits_{\substack{i\in I_{2}\ \mathrm{such\ that}%
\\x_{2}^{\prime}(i)=0}}k(i), \label{11}%
\end{equation}%
\begin{equation}
v_{p}\left(  \mathrm{Nm}_{\varphi}(\vec{c})\right)  \geq0, \label{12}%
\end{equation}%
\begin{align}
v_{p}\left(  \mathrm{Nm}_{\varphi}\left(  \vec{a}\right)  \right)
+v_{p}\left(  \mathrm{Nm}_{\varphi}(\vec{b})\right)  \geq\sum
\limits_{\substack{i\in I_{1}\ \mathrm{such\ that}\\x_{2}^{\prime\prime
}(i)\not =0}}k_{1}(i)+\sum\limits_{\substack{i\in I_{1}\ \mathrm{such\ that}%
\\x_{2}^{\prime\prime}(i)=0\ \mathrm{and}\text{\ }x_{2}^{\prime}(i)=1}%
}k_{2}(i)+  & \label{q}\\
& \nonumber\\
+\sum\limits_{\substack{i\in I_{1}\ \mathrm{such\ that}\text{ }\\x_{2}%
(i)=x_{2}^{\prime}(i)=0}}\left(  k_{1}(i)+k_{2}(i)\right)  +\sum
\limits_{\substack{i\in I_{2}\ \mathrm{such\ that}\text{\ }x_{2}^{\prime
}(i)=1\ \mathrm{or}\\x_{2}^{\prime}(i)=0\ \mathrm{and}\text{ }x_{2}(i)\neq
0}}k(i)+\sum\limits_{\substack{i\in I_{2}\ \mathrm{such\ that}\\x_{2}%
(i)=x_{2}^{\prime}(i)=0}}2k(i)+\sum\limits_{\substack{i\in I_{3}%
\ \mathrm{such\ that}\\\ x_{2}^{\prime\prime}(i)=0}}k(i),  & \nonumber
\end{align}%
\begin{align}
v_{p}\left(  \mathrm{Nm}_{\varphi}\left(  \vec{a}\right)  \right)
+v_{p}\left(  \mathrm{Nm}_{\varphi}(\vec{c})\right)  \geq\sum
\limits_{\substack{i\in I_{1}\ \mathrm{such\ that}\\x_{1}(i)\not =0}%
}k_{1}(i)+\sum\limits_{\substack{i\in I_{1}\ \mathrm{such\ that}\\\text{
}x_{1}(i)=0}}k_{2}(i)+  & \label{13}\\
& \nonumber\\
+\sum\limits_{i\in I_{2}}k(i)+\sum\limits_{\substack{i\in I_{3}%
\ \mathrm{such\ that}\\x_{1}(i)=0}}k(i),  & \nonumber
\end{align}%
\begin{equation}
v_{p}\left(  \mathrm{Nm}_{\varphi}(\vec{b})\right)  +v_{p}\left(
\mathrm{Nm}_{\varphi}(\vec{c})\right)  \geq\sum\limits_{i\in I_{1}}%
k_{1}(i)+\sum\limits_{i\in I_{2}}k(i). \label{16}%
\end{equation}
The filtered $\varphi$-module is irreducible in $\mathrm{MF}_{K_{f}%
,E}^{\varphi,w.a}$ if and only if all inequalities $\left(  \ref{10}\right)
$-$\left(  \ref{16}\right)  $ are strict. Assuming that $\left(
D,\varphi\right)  $ is weakly admissible,

\begin{enumerate}
\item[(a)] The submodule $D_{0}$ is weakly admissible if and only if
inequality $\left(  \ref{10}\right)  $ is equality.

\item[(b)] The submodule $D_{1}$ is weakly admissible if and only if
inequality $\left(  \ref{11}\right)  $ is equality.

\item[(c)] The submodule $D_{2}$ is weakly admissible if and only if
inequality $\left(  \ref{12}\right)  $ is equality.

\item[(d)] The submodule $D_{01}$ is weakly admissible if and only if
inequality $\left(  \ref{q}\right)  $ is equality.

\item[(e)] The submodule $D_{02}$ is weakly admissible if and only if
inequality $\left(  \ref{13}\right)  $ is equality.

\item[(f)] The submodule $D_{12}$ is weakly admissible if and only if
inequality $\left(  \ref{16}\right)  $ is equality.
\end{enumerate}
\end{proposition}

\noindent We denote the filtered module of Proposition \ref{phi, fil} by
$\left(  \mathcal{D}\left(  \underline{a},\underline{x}\right)  \right)  ,$
where $\underline{a}=\left(  \vec{a},\vec{b},\vec{c}\right)  $ and
$\underline{x}=\left(  \vec{x}_{1},\vec{x}_{2},\vec{x}_{2}^{\prime},\vec
{x}_{2}^{\prime\prime}\right)  .$

\begin{proposition}
\label{iso}We have $\left(  \mathcal{D}\left(  \underline{a},\underline
{x}\right)  \right)  \simeq\left(  \mathcal{D}\left(  \underline{a_{1}%
},\underline{y}\right)  \right)  ,$ where $\underline{a_{1}}=\left(  \vec
{a}_{1},\vec{b}_{1},\vec{c}_{1}\right)  $ and $\underline{y}=\left(  \vec
{y}_{1},\vec{y}_{2},\vec{y}_{2}^{\prime},\vec{y}_{2}^{\prime\prime}\right)  $
if and only if either

\begin{enumerate}
\item $\left(  \mathrm{Nm}_{\varphi}(\vec{a}),\mathrm{Nm}_{\varphi}(\vec
{b}),\mathrm{Nm}_{\varphi}(\vec{c})\right)  =\left(  \mathrm{Nm}_{\varphi
}(\vec{a}_{1}),\mathrm{Nm}_{\varphi}(\vec{b}_{1}),\mathrm{Nm}_{\varphi}%
(\vec{c}_{1})\right)  ,$ and

\begin{enumerate}
\item For all $i$ such that the filtration of $D_{i}$ is given by formula
(\ref{1}), either

\begin{enumerate}
\item $x_{2}(i)=y_{2}(i)=x_{2}^{\prime}(i)=y_{2}^{\prime}(i)=1,\ $%
and$\ x_{1}(i)=y_{1}(i)$ or

\item $x_{2}(i)=y_{2}(i)=0,\ x_{2}^{\prime}(i)=y_{2}^{\prime}(i)=1,\ $%
and\ $x_{1}(i)\not =0\ $if\ and\ only\ if\ $y_{1}(i)\neq0,$ or

\item $x_{2}^{\prime}(i)=y_{2}^{\prime}(i)=0,\ x_{2}(i)=1\ $if and only
if\ $y_{2}(i)=1,\ $and\ $x_{1}(i)\neq0\ $if\ and\ only\ if $y_{1}(i)\neq0.$
\end{enumerate}

\item For all $i$ such that the filtration of $D_{i}$ is given by formula
(\ref{fil1}), $x_{2}^{\prime}(i)=y_{2}^{\prime}(i),\ $and $x_{2}(i)=y_{2}(i).$

\item For all $i$ such that the filtration of $D_{i}$ is given by formula
(\ref{filtr'}), $x_{1}(i)=y_{1}(i)$ and $x_{2}^{\prime\prime}(i)=y_{2}%
^{\prime\prime}(i).$
\end{enumerate}

\item $\left(  \mathrm{Nm}_{\varphi}(\vec{a}),\mathrm{Nm}_{\varphi}(\vec
{b}),\mathrm{Nm}_{\varphi}(\vec{c})\right)  =\left(  \mathrm{Nm}_{\varphi
}(\vec{b}_{1}),\mathrm{Nm}_{\varphi}(\vec{a}_{1}),\mathrm{Nm}_{\varphi}%
(\vec{c}_{1})\right)  ,$ and

\begin{enumerate}
\item For all $i$ such that the filtration of $D_{i}$ is given by formula
(\ref{1}), either

\begin{enumerate}
\item \bigskip$x_{2}=y_{2}=x_{2}^{\prime}=y_{2}^{\prime}=0,\ $and $x_{1}%
y_{1}\neq0,$ or

\item $x_{2}^{\prime}=y_{2}=0,\ x_{2}=y_{2}^{\prime}=1,\ $and$\ y_{1}x_{1}%
\neq0,$ or

\item $x_{2}=y_{2}^{\prime}=0,\ x_{2}^{\prime}=y_{2}=1,\ $and $x_{1}y_{1}%
\neq0.$
\end{enumerate}

\item For all $i$ such that the filtration of $D_{i}$ is given by formula
(\ref{fil1}), $x_{2}^{\prime}(i)=y_{2}(i),$ and $y_{2}^{\prime}(i)=x_{2}(i).$

\item For all $i$ such that the filtration of $D_{i}$ is given by formula
(\ref{filtr'}), $x_{1}(i)=y_{1}(i)=1$ and $x_{2}^{\prime\prime}(i)=y_{2}%
^{\prime\prime}(i).$
\end{enumerate}

\item $\left(  \mathrm{Nm}_{\varphi}(\vec{a}),\mathrm{Nm}_{\varphi}(\vec
{b}),\mathrm{Nm}_{\varphi}(\vec{c})\right)  =\left(  \mathrm{Nm}_{\varphi
}(\vec{c}_{1}),\mathrm{Nm}_{\varphi}(\vec{a}_{1}),\mathrm{Nm}_{\varphi}%
(\vec{b}_{1})\right)  ,$ and

\begin{enumerate}
\item For all $i$ such that the filtration of $D_{i}$ is given by formula
(\ref{1}), either

\begin{enumerate}
\item $x_{2}^{\prime}(i)=y_{2}(i)=0,\ x_{2}(i)=y_{2}^{\prime}(i)=1,\ $and
$x_{1}(i)y_{1}(i)\neq0,$ or

\item $x_{2}(i)=y_{2}(i)=x_{2}^{\prime}(i)=y_{2}^{\prime}(i)=1,\ $and
$x_{1}(i)y_{1}(i)+x_{1}(i)+1=0.$
\end{enumerate}

\item For all $i$ such that the filtration of $D_{i}$ is given by formula
(\ref{fil1}), $x_{2}(i)=y_{2}^{\prime}(i)=1,$ and $x_{2}^{\prime}(i)=1$ if and
only if $y_{2}(i)=1.$

\item For all $i$ such that the filtration of $D_{i}$ is given by formula
(\ref{filtr'}), $x_{1}(i)=y_{2}^{\prime\prime}(i)=1$ and $x_{2}^{\prime\prime
}(i)=1$ if and only if $y_{1}(i)=1.$
\end{enumerate}

\item $\left(  \mathrm{Nm}_{\varphi}(\vec{a}),\mathrm{Nm}_{\varphi}(\vec
{b}),\mathrm{Nm}_{\varphi}(\vec{c})\right)  =\left(  \mathrm{Nm}_{\varphi
}(\vec{a}_{1}),\mathrm{Nm}_{\varphi}(\vec{c}_{1}),\mathrm{Nm}_{\varphi}%
(\vec{b}_{1})\right)  ,$ and

\begin{enumerate}
\item For all $i$ such that the filtration of $D_{i}$ is given by formula
(\ref{1}), either

\begin{enumerate}
\item $x_{2}(i)=y_{2}(i)=y_{2}^{\prime}(i)=x_{2}^{\prime}(i)=1,\ $%
and$\ x_{1}(i)+y_{1}(i)+1=0,$ or

\item $x_{2}(i)=y_{2}(i)=0,\ x_{2}^{\prime}(i)=y_{2}^{\prime}(i)=1,\ $%
and\ $x_{1}(i)\neq0\ $if\ and\ only\ if\ $y_{1}(i)\neq0.$
\end{enumerate}

\item For all $i$ such that the filtration of $D_{i}$ is given by formula
(\ref{fil1}), $x_{2}^{\prime}(i)=y_{2}^{\prime}(i)=1$ and$\ x_{2}%
(i)=y_{2}(i).$

\item For all $i$ such that the filtration of $D_{i}$ is given by formula
(\ref{filtr'}), $x_{2}^{\prime\prime}(i)=y_{1}(i)$ and $x_{1}(i)=y_{2}%
^{\prime\prime}(i).$
\end{enumerate}

\item $\left(  \mathrm{Nm}_{\varphi}(\vec{a}),\mathrm{Nm}_{\varphi}(\vec
{b}),\mathrm{Nm}_{\varphi}(\vec{c})\right)  =\left(  \mathrm{Nm}_{\varphi
}(\vec{c}_{1}),\mathrm{Nm}_{\varphi}(\vec{b}_{1}),\mathrm{Nm}_{\varphi}%
(\vec{a}_{1})\right)  ,$ and

\begin{enumerate}
\item For all $i$ such that the filtration of $D_{i}$ is given by formula
(\ref{1}), either

\begin{enumerate}
\item $x_{2}(i)=y_{2}(i)=x_{2}^{\prime}(i)=y_{2}^{\prime}(i)=1,\ x_{1}%
(i)\neq0,\ $and $x_{1}(i)\left(  y_{1}(i)+1\right)  +y_{1}(i)=0,$ or

\item $x_{1}(i)=y_{1}(i)=0,\ $and $x_{2}(i)=y_{2}(i)=x_{2}^{\prime}%
(i)=y_{2}^{\prime}(i)=1,$ or

\item $x_{1}(i)=y_{1}(i)=x_{2}^{\prime}(i)=y_{2}^{\prime}(i)=0,\ $and
$x_{2}(i)=y_{2}(i)=1,$ or

\item $x_{2}^{\prime}(i)=y_{2}^{\prime}(i)=0,\ x_{2}(i)=y_{2}(i)=1,$ and
$x_{1}(i)y_{1}(i)\neq0.$
\end{enumerate}

\item For all $i$ such that the filtration of $D_{i}$ is given by formula
(\ref{fil1}), $x_{2}(i)=y_{2}(i)=1$ and $y_{2}^{\prime}(i)=x_{2}^{\prime}(i).$

\item For all $i$ such that the filtration of $D_{i}$ is given by formula
(\ref{filtr'}), $x_{2}^{\prime\prime}(i)=y_{2}^{\prime\prime}(i)=1$ and
$x_{1}(i)=y_{1}(i).$
\end{enumerate}

\item $\left(  \mathrm{Nm}_{\varphi}(\vec{a}),\mathrm{Nm}_{\varphi}(\vec
{b}),\mathrm{Nm}_{\varphi}(\vec{c})\right)  =\left(  \mathrm{Nm}_{\varphi
}(\vec{b}_{1}),\mathrm{Nm}_{\varphi}(\vec{c}_{1}),\mathrm{Nm}_{\varphi}%
(\vec{a}_{1})\right)  ,$ and

\begin{enumerate}
\item For all $i$ such that the filtration of $D_{i}$ is given by formula
(\ref{1}), either

\begin{enumerate}
\item $x_{2}(i)=y_{2}(i)=x_{2}^{\prime}(i)=y_{2}^{\prime}(i)=1,\ $and
$y_{1}(i)\left(  x_{1}(i)+1\right)  +1=0,$ or

\item $x_{2}^{\prime}(i)=y_{2}(i)=1,\ x_{2}(i)=y_{2}^{\prime}(i)=0,\ $and
$x_{1}(i)y_{1}(i)\neq0.$\bigskip
\end{enumerate}

\item For all $i$ such that the filtration of $D_{i}$ is given by formula
(\ref{fil1}), $x_{2}^{\prime}(i)=y_{2}(i)=1$ and $y_{2}^{\prime}(i)=x_{2}(i).$

\item For all $i$ such that the filtration of $D_{i}$ is given by formula
(\ref{filtr'}), $x_{2}^{\prime\prime}(i)=y_{1}(i)=1$ and $x_{1}(i)=y_{2}%
^{\prime\prime}(i).$\bigskip\bigskip
\end{enumerate}
\end{enumerate}
\end{proposition}

\begin{proposition}
\label{semistable}Let $\left(  D,\varphi,N\right)  $ be a filtered $\left(
\varphi,N\right)  $-module over $K$ with $E$ coefficients and distinct
eigenvalues of frobenius. There exists some ordered bases $\underline{e}$ of
$D$ over $E^{\mid\mathcal{\tau}\mid}$ such that $\mathrm{Mat}_{\underline{e}%
}\left(  \varphi\right)  =\mathrm{diag}\left(  \vec{a},\vec{b},\vec{c}\right)
$ with distinct $\mathrm{Nm}_{\varphi}(\vec{a}),\mathrm{Nm}_{\varphi}(\vec
{b}),\mathrm{Nm}_{\varphi}(\vec{c})$ and filtration as in Proposition
\ref{phi, fil}. The matrix $[N]_{\underline{e}}$ of the monodromy operator
with respect to the basis $\underline{e}$ has one of the following forms with
\underline{at most} two nonzero entries:

\begin{enumerate}
\item
\[
\lbrack N]_{\underline{e}}=\left(
\begin{array}
[c]{ccc}%
\vec{0} & \vec{a}_{12} & \vec{0}\\
\vec{0} & \vec{0} & \vec{a}_{23}\\
\vec{a}_{31} & \vec{0} & \vec{0}%
\end{array}
\right)  ,
\]
or

\item
\[
\lbrack N]_{\underline{e}}=\left(
\begin{array}
[c]{ccc}%
\vec{0} & \vec{0} & \vec{a}_{13}\\
\vec{a}_{21} & \vec{0} & \vec{0}\\
\vec{0} & \vec{a}_{32} & \vec{0}%
\end{array}
\right)  ,
\]
or

\item
\[
\lbrack N]_{\underline{e}}=\left(
\begin{array}
[c]{ccc}%
\vec{0} & \vec{a}_{12} & \vec{0}\\
\vec{0} & \vec{0} & \vec{a}_{23}\\
\vec{a}_{31} & \vec{0} & \vec{0}%
\end{array}
\right)  ,
\]

\end{enumerate}

where%
\begin{align*}
\vec{a}_{12}  &  =a_{12}\left(  1,\frac{b\left(  0\right)  }{pa\left(
0\right)  },\frac{b\left(  0\right)  b\left(  1\right)  }{p^{2}a\left(
0\right)  a\left(  1\right)  },\cdots,\frac{b\left(  0\right)  b\left(
1\right)  \cdots b\left(  f-2\right)  }{p^{f-1}a\left(  0\right)  a\left(
1\right)  \cdots a\left(  f-2\right)  }\right)  \ \text{for some }a_{12}\in
E,\\
\vec{a}_{13}  &  =a_{13}\left(  1,\frac{c\left(  0\right)  }{pa\left(
0\right)  },\frac{c\left(  0\right)  c\left(  1\right)  }{p^{2}a\left(
0\right)  a\left(  1\right)  },\cdots,\frac{c\left(  0\right)  c\left(
1\right)  \cdots c\left(  f-2\right)  }{p^{f-1}a\left(  0\right)  a\left(
1\right)  \cdots a\left(  f-2\right)  }\right)  \ \text{for some }a_{13}\in
E,\\
\vec{a}_{21}  &  =a_{21}\left(  1,\frac{a\left(  0\right)  }{pb\left(
0\right)  },\frac{a\left(  0\right)  a\left(  1\right)  }{p^{2}b\left(
0\right)  b\left(  1\right)  },\cdots,\frac{a\left(  0\right)  a\left(
1\right)  \cdots a\left(  f-2\right)  }{p^{f-1}b\left(  0\right)  b\left(
1\right)  \cdots b\left(  f-2\right)  }\right)  \ \text{for some }a_{21}\in
E,\\
\vec{a}_{23}  &  =a_{23}\left(  1,\frac{c\left(  0\right)  }{pb\left(
0\right)  },\frac{c\left(  0\right)  c\left(  1\right)  }{p^{2}b\left(
0\right)  b\left(  1\right)  },\cdots,\frac{c\left(  0\right)  c\left(
1\right)  \cdots c\left(  f-2\right)  }{p^{f-1}b\left(  0\right)  b\left(
1\right)  \cdots b\left(  f-2\right)  }\right)  \ \text{for some }a_{23}\in
E,\\
\vec{a}_{31}  &  =a_{31}\left(  1,\frac{a\left(  0\right)  }{pc\left(
0\right)  },\frac{a\left(  0\right)  a\left(  1\right)  }{p^{2}c\left(
0\right)  c\left(  1\right)  },\cdots,\frac{a\left(  0\right)  a\left(
1\right)  \cdots a\left(  f-2\right)  }{p^{f-1}c\left(  0\right)  c\left(
1\right)  \cdots c\left(  f-2\right)  }\right)  \ \text{for some }a_{31}\in
E,\\
\vec{a}_{32}  &  =a_{32}\left(  1,\frac{b\left(  0\right)  }{pc\left(
0\right)  },\frac{b\left(  0\right)  b\left(  1\right)  }{p^{2}c\left(
0\right)  c\left(  1\right)  },\cdots,\frac{b\left(  0\right)  b\left(
1\right)  \cdots b\left(  f-2\right)  }{p^{f-1}c\left(  0\right)  c\left(
1\right)  \cdots c\left(  f-2\right)  }\right)  \ \text{for some }a_{32}\in E.
\end{align*}
If $a_{12}\neq0$ then $\mathrm{Nm}_{\varphi}(\vec{b})=p^{f}\mathrm{Nm}%
_{\varphi}(\vec{a}),$ if $a_{31}\neq0$ then $\mathrm{Nm}_{\varphi}(\vec
{a})=p^{f}\mathrm{Nm}_{\varphi}(\vec{c})$ and if $a_{23}\neq0$ then
$\mathrm{Nm}_{\varphi}(\vec{c})=p^{f}\mathrm{Nm}_{\varphi}(\vec{b}),$ if
$a_{13}\neq0$ then $\mathrm{Nm}_{\varphi}(\vec{c})=p^{f}\mathrm{Nm}_{\varphi
}(\vec{a}),$ if $a_{32}\neq0$ then $\mathrm{Nm}_{\varphi}(\vec{b}%
)=p^{f}\mathrm{Nm}_{\varphi}(\vec{c}),$ and if $a_{21}\neq0$ then
$\mathrm{Nm}_{\varphi}(\vec{c})=p^{f}\mathrm{Nm}_{\varphi}(\vec{b}).$
\end{proposition}

Since the eigenvalues of frobenius are distinct, in all cases of Proposition
\ref{semistable}, at most two of the entries of the matrix of $[N]_{\underline
{e}}$ are nonzero.

\subsection{Proof of the propositions}

The map $\varphi^{f}$ is $E^{\mid\mathcal{\tau}\mid}$-linear, $D_{i}\ $is
$3$-dimensional over $E$ and $D_{i}\ $is $\varphi^{f}$-stable for all $i.$ By
Jordan decomposition, for each $i=0,1,...,f-1,$ there exists an ordered basis
$\underline{e}(i)=\left(  e_{0}(i),e_{1}(i),e_{2}(i)\right)  $ of $D_{i}$ over
$E$ such that $\mathrm{Mat}_{\underline{e}^{i}}\left(  \varphi^{f}\right)  $
has one of the following forms:%
\[
\mathrm{Mat}_{\underline{e}^{i}}\left(  \varphi^{f}\right)  =\mathrm{diag}%
\left(  a_{11}(i),a_{22}(i),a_{33}(i)\right)
\]
for some $a_{jj}^{i}\in E^{\times},$ or%
\[
\mathrm{Mat}_{\underline{e}^{i}}\left(  \varphi^{f}\right)  =\left(
\begin{array}
[c]{ccc}%
a_{11}(i) & 0 & 0\\
1 & a_{11}(i) & 0\\
0 & 0 & a_{33}(i)
\end{array}
\right)
\]
for some distinct $a_{11}(i),a_{33}(i)\in E^{\times},$ or%
\[
\mathrm{Mat}_{\underline{e}^{i}}\left(  \varphi^{f}\right)  =\left(
\begin{array}
[c]{ccc}%
a_{11}(i) & 0 & 0\\
1 & a_{11}(i) & 0\\
0 & 1 & a_{11}(i)
\end{array}
\right)
\]
for some $a_{11}(i)\in E^{\times}.$ The $E$-linear map $\varphi:D_{i}%
\rightarrow D_{i+1}$ is an isomorphism and $\{a_{11}(i),a_{22}(i),a_{33}%
(i)\}=\{a_{11}(i+1),a_{22}(i+1),a_{33}(i+1)\}$ for all $i.$ Permuting the
basis elements of the $D_{i}$ we may assume that $a_{11}:=a_{11}(i),$
$a_{22}:=a_{22}(i)$ and $a_{22}:=a_{33}(i)$ for all $i.$ Moreover, if
$\mathrm{Mat}_{\underline{\eta}^{i}}\left(  \varphi^{f}\right)  $ has one of
the types above for some $i$ then $\mathrm{Mat}_{\underline{\eta}^{i}}\left(
\varphi^{f}\right)  $ has the same type for all $i.$ For the ordered basis
$\underline{e}=\left(  e_{0},e_{1},e_{2}\right)  ,$ where $e_{j}:=%
%TCIMACRO{\tsum \limits_{i=0}^{f-1}}%
%BeginExpansion
{\textstyle\sum\limits_{i=0}^{f-1}}
%EndExpansion
e_{j}(i),$ the shape of $\mathrm{Mat}_{\underline{e}}\left(  \varphi
^{f}\right)  $ is of the forms: either%
\begin{equation}
\mathrm{Mat}_{\underline{e}}\left(  \varphi^{f}\right)  =\mathrm{diag}\left(
a_{11}\cdot\vec{1},a_{22}\cdot\vec{1},a_{33}\cdot\vec{1}\right)  \tag{$%
1$}\label{phi 1}%
\end{equation}
or%
\begin{equation}
\mathrm{Mat}_{\underline{e}}\left(  \varphi^{f}\right)  =\left(
\begin{array}
[c]{ccc}%
a_{11}\cdot\vec{1} & \vec{0} & \vec{0}\\
\vec{1} & a_{11}\cdot\vec{1} & \vec{0}\\
\vec{0} & \vec{0} & a_{33}\cdot\vec{1}%
\end{array}
\right)  , \tag{$2$}%
\end{equation}
or%
\begin{equation}
\mathrm{Mat}_{\underline{e}}\left(  \varphi^{f}\right)  =\left(
\begin{array}
[c]{ccc}%
a_{11}\cdot\vec{1} & \vec{0} & \vec{0}\\
\vec{1} & a_{11}\cdot\vec{1} & \vec{0}\\
\vec{0} & \vec{1} & a_{11}\cdot\vec{1}%
\end{array}
\right)  . \tag{$3$}%
\end{equation}

\begin{proposition}
\label{dist eigen}A\ rank $3$ filtered $\varphi$-module $\left(
D,\varphi\right)  $ over $K$ with $E$-coefficients has distinct eigenvalues of
frobenius if and only if there exists an ordered basis$\ \underline{e}$ such
that%
\[
\mathrm{Mat}_{\underline{e}}\left(  \varphi^{f}\right)  =\mathrm{diag}\left(
a_{11}\cdot\vec{1},a_{22}\cdot\vec{1},a_{33}\cdot\vec{1}\right)
\]
for some distinct $a_{jj}\in E^{\times}.$
\end{proposition}

\begin{proof}
Let $P:=\mathrm{Mat}_{\underline{e}}\left(  \varphi\right)  =\left(
P_{0},P_{1},...,P_{f-1}\right)  \ $and $Q:=\mathrm{Mat}_{\underline{e}}\left(
\varphi^{f}\right)  =\left(  Q_{0},Q_{1},...,Q_{f-1}\right)  .$ Since
$Q=\mathrm{Nm}_{\varphi}\left(  P\right)  \ $and $\varphi^{f}\left(  P\right)
=P,$ we have $Q=P\varphi\left(  Q\right)  P^{-1}$ and $Q_{i}=P_{i}Q_{i+1}%
P_{i}^{-1}$ for all $i=1,2,...,f.$ Since $Q_{i}=Q_{i+1}=\mathrm{diag}\left(
a_{11},a_{22},a_{33}\right)  $ for all $i,$ we have $Q_{i}P_{i}=P_{i}Q_{i}%
,\ $and since the $a_{jj}$ are distinct, a direct computation shows that
$P_{i}\ $is diagonal for all $i.$ The proposition follows since $\mathrm{Nm}%
_{\varphi}\left(  \vec{a}\right)  =a_{11}\cdot\vec{1},$ $\mathrm{Nm}_{\varphi
}(\vec{b})=a_{22}\cdot\vec{1},$ $\mathrm{Nm}_{\varphi}\left(  \vec{c}\right)
=a_{33}\cdot\vec{1},$ and the $a_{jj}$ are distinct. The other direction is trivial.
\end{proof}

\begin{proof}
[Proof of Proposition \ref{phi, fil}]Let $\underline{e}=\left(  e_{0}%
,e_{1},e_{2}\right)  $ be an ordered basis of $D$ such that $\mathrm{Mat}%
_{\underline{e}}\left(  \varphi^{f}\right)  $ is diagonal, and fix some
$i\in\{0,1,...,f-1\}.$ Assume that the labeled Hodge-Tate weights of $D$ with
respect to the embedding $\tau_{i}$ are distinct. Then the filtration
$\mathrm{Fil}^{\mathit{j}}\left(  D_{i}\right)  :=e_{\tau_{i}}\mathrm{Fil}%
^{\mathit{j}}\left(  D\right)  $ of $D_{i}:=e_{\tau_{i}}D$ has the form
\[
\mathrm{Fil}^{j}D_{i}=\left\{
\begin{array}
[c]{l}%
e_{\tau_{i}}D\ \ \mathrm{if}\text{ }j\leq0,\\
e_{\tau_{i}}D_{2}\ \mathrm{if}\text{ }1\leq j\leq k_{1}(i),\\
e_{\tau_{i}}D_{1}\ \mathrm{if}\text{ }1+k_{1}\leq j\leq k_{2}(i)\\
\ 0\ \ \ \ \ \ \mathrm{if\ }j\geq1+k_{2}(i),
\end{array}
\right.
\]
where%
\[
e_{\tau_{i}}D_{2}=e_{\tau_{i}}E^{\mid\mathcal{\tau}\mid}\left(  \vec{u}%
_{0}^{i}e_{0}+\vec{u}_{1}^{i}e_{1}+\vec{u}_{2}^{i}e_{2}\right)  \oplus
e_{\tau_{i}}E^{\mid\mathcal{\tau}\mid}\left(  \vec{v}_{0}^{i}e_{0}+\vec{v}%
_{1}^{i}e_{1}+\vec{v}_{2}^{i}e_{2}\right)  ,
\]
with $\vec{u}_{j}^{i},\vec{v}_{j}^{i}\in E^{\mid\mathcal{\tau}\mid}.$ The
vectors $e_{\tau_{i}}\left(  \vec{u}_{0}^{i}e_{0}+\vec{u}_{1}^{i}e_{1}+\vec
{u}_{2}^{i}e_{2}\right)  \ $and $e_{\tau_{i}}\left(  \vec{v}_{0}^{i}e_{0}%
+\vec{v}_{1}^{i}e_{1}+\vec{v}_{2}^{i}e_{2}\right)  $ are linearly independent
over $E$ and for simplicity we write $\vec{u}_{j}:=\vec{u}_{j}^{i}\ $and
$\vec{v}_{j}=\vec{v}_{j}^{i}.$ The space $e_{\tau_{i}}D_{1}$ is some
$1$-dimensional subspace\ of $e_{\tau_{i}}D_{2}.$ Applying an automorphism of
$\left(  D,\varphi\right)  \ $which permutes the basis elements, if necessary,
we may assume that the $i$-th coordinate $v_{0}(i)\ $of $\vec{v}_{0}\ $is
nonzero. Then
\[
e_{\tau_{i}}D_{2}=e_{\tau_{i}}E^{\mid\mathcal{\tau}\mid}\left(  u(i)-u_{0}%
(i)v_{0}(i)^{-1}v(i)\right)  \oplus e_{\tau_{i}}E^{\mid\mathcal{\tau}\mid
}v_{0}(i)^{-1}v(i).
\]
Since $e_{\tau_{i}}D_{2}$ is $2$-dimensional over $E,$ we have $u_{1}%
(i)-u_{0}(i)\frac{v_{1}(i)}{v_{0}(i)}\neq0$ or $u_{2}(i)-u_{0}(i)\frac
{v_{2}(i)}{v_{0}(i)}\neq0.$ Applying the automorphism of $\left(
D,\varphi\right)  $ which permutes the basis elements $e_{1}$ and $e_{2}\ $and
fixes $e_{0},\ $if necessary, we may assume that $r(i):=u_{1}(i)-u_{0}%
(i)\frac{v_{1}(i)}{v_{0}(i)}\neq0.$ Let%
\[
w(i):=e_{1}+\allowbreak\frac{u_{0}(i)v_{2}(i)-v_{0}(i)u_{2}(i)}{u_{0}%
(i)v_{1}(i)-u_{1}(i)v_{0}(i)}e_{2},
\]
then%
\begin{align*}
e_{\tau_{i}}D_{2}  &  =e_{\tau_{i}}E^{\mid\mathcal{\tau}\mid}r(i)^{-1}\left(
u(i)-u_{0}(i)v_{0}(i)^{-1}v(i)\right)  \oplus e_{\tau_{i}}E^{\mid
\mathcal{\tau}\mid}v_{0}(i)^{-1}v(i)\\
&  =e_{\tau_{i}}E^{\mid\mathcal{\tau}\mid}w(i)\oplus e_{\tau_{i}}%
E^{\mid\mathcal{\tau}\mid}v_{0}(i)^{-1}v(i)=e_{\tau_{i}}E^{\mid\mathcal{\tau
}\mid}w(i)\oplus e_{\tau_{i}}E^{\mid\mathcal{\tau}\mid}\left(  v_{0}%
(i)^{-1}v(i)-v_{1}(i)v_{0}(i)^{-1}w(i)\right) \\
&  =e_{\tau_{i}}E^{\mid\mathcal{\tau}\mid}\left(  e_{1}+\allowbreak\frac
{u_{0}(i)v_{2}(i)-v_{0}(i)u_{2}(i)}{u_{0}(i)v_{1}(i)-u_{1}(i)v_{0}(i)}%
e_{2}\right)  \oplus e_{\tau_{i}}E^{\mid\mathcal{\tau}\mid}\left(  e_{0}%
+\frac{u_{2}(i)v_{1}(i)-u_{1}(i)v_{2}(i)}{u_{0}(i)v_{1}(i)-u_{1}(i)v_{0}%
(i)}e_{2}\right)  .
\end{align*}
Let%
\[
y_{2}(i)=\frac{u_{2}(i)v_{1}(i)-u_{1}(i)v_{2}(i)}{u_{0}(i)v_{1}(i)-u_{1}%
(i)v_{0}(i)}\ \text{and\ }y_{2}^{\prime}(i)=\frac{u_{0}(i)v_{2}(i)-v_{0}%
(i)u_{2}(i)}{u_{0}(i)v_{1}(i)-u_{1}(i)v_{0}(i)}.
\]
We have $e_{\tau_{i}}D_{2}=e_{\tau_{i}}E^{\mid\mathcal{\tau}\mid}\left(
e_{0}+y_{2}(i)e_{2}\right)  \oplus e_{\tau_{i}}E^{\mid\mathcal{\tau}\mid
}\left(  e_{1}+y_{2}^{\prime}(i)e_{2}\right)  ,$ and let%
\[
e_{\tau_{i}}D_{1}=e_{\tau_{i}}E^{\mid\mathcal{\tau}\mid}\left(  \lambda
(i)e_{0}+\mu(i)e_{1}+\left(  \lambda(i)y_{2}(i)+\mu(i)y_{2}^{\prime
}(i)\right)  e_{2}\right)
\]
for some $\lambda(i),\mu(i)\in E$ with$\ \left(  \lambda(i),\mu(i)\right)
\neq\left(  0,0\right)  .$ If $\lambda(i)\neq0,$ we let $x_{1}(i)=\frac
{\mu(i)}{\lambda(i)},\ x_{2}(i)=y_{2}(i)\ $and $x_{2}^{\prime}(i)=y_{2}%
^{\prime}(i),$ otherwise we apply the automorphism of $\left(  D,\varphi
\right)  $ which permutes $e_{0}$ with $e_{1}$ and fixes $e_{2}.$ Then%
\[
e_{\tau_{i}}D_{2}=e_{\tau_{i}}E^{\mid\mathcal{\tau}\mid}\left(  e_{0}%
+x_{2}(i)e_{2}\right)  \oplus e_{\tau_{i}}E^{\mid\mathcal{\tau}\mid}\left(
e_{1}+x_{2}^{\prime}(i)e_{2}\right)  ,
\]
where $x_{2}(i)=y_{2}^{\prime}(i)$ and $x_{2}^{\prime}=y_{2},$%
\[
e_{\tau_{i}}D_{1}=e_{\tau_{i}}E^{\mid\mathcal{\tau}\mid}\left(  e_{0}%
+\frac{\lambda(i)}{\mu(i)}e_{1}+\left(  \frac{\lambda(i)}{\mu(i)}x_{2}%
^{\prime}(i)+x_{2}(i)\right)  e_{2}\right)
\]
and we let $x_{1}(i)=\frac{\lambda(i)}{\mu(i)}.$ If $x_{2}^{\prime}(i)\neq0,$
applying the automorphism which maps $e_{2}$ to $\left(  x_{2}^{\prime
}(i)\right)  ^{-1}\cdot\vec{1}\cdot e_{2}$ and fixes $e_{0}$ and $e_{1}$ we
may assume that $x_{2}^{\prime}(i)=1.$ Similarly, if $x_{2}(i)\neq0,$ applying
the automorphism of $\left(  D,\varphi\right)  $ which maps $e_{0}$ to
$x_{2}(i)\cdot\vec{1}\cdot e_{0}$ and fixes$\ e_{1}$ and $e_{2},$ we may
assume that $x_{2}(i)=1.$ The matrix of $\varphi$ remains diagonal after
applying each of the automorphisms of $\left(  D,\varphi\right)  \ $above, and
its eigenvalues remain distinct. This concludes the proof of the first part of
the proposition in the case of thee distinct labeled Hodge-Tate weights with
respect to the embedding $\tau_{i}.$ The proofs for the cases of two or three
equal Hodge-Tate weights are special cases of the proof above.
\end{proof}

\noindent

\begin{proof}
[Proof of Proposition \ref{wa}]If $D^{\ast}=E^{\mid\mathcal{\tau}\mid}\left(
\vec{z}_{0}e_{0}+\vec{z}_{1}e_{1}+\vec{z}_{2}e_{2}\right)  $ is a rank one
$\varphi$-stable submodule, prove that if $\vec{z}_{r}\neq\vec{0}$ for some
$r,$ then $\vec{z}_{j}=0$ for all $j\neq r.$ Indeed, let $\vec{z}_{0}\neq
\vec{0}$ and let $z_{0}(i)\neq0$ for some $i.$ We may assume that
$z_{0}\left(  i\right)  =1.$ Since $D^{\ast}$ is $\varphi$-stable,
$z_{0}\left(  i\right)  \neq0$ for all $i.$ We apply the linear map
$\varphi^{f}$ to $e_{\tau_{i}}D^{\ast}$ and we have%
\[
e_{\tau_{i}}\left(  \alpha e_{0}+\beta z_{1}\left(  i\right)  e_{1}+\gamma
z_{2}(i)e_{2}\right)  =te_{\tau_{i}}\left(  e_{0}+z_{1}\left(  i\right)
e_{1}+z_{2}\left(  i\right)  e_{2}\right)  .
\]
Since the $\alpha,\beta,\gamma$ are distinct this implies that $z_{1}\left(
i\right)  =z_{2}\left(  i\right)  =0\ $for all $i.$ Hence the only rank one
$\varphi$-stable submodules of $D$ are the $D_{s}=E^{\mid\mathcal{\tau}\mid
}e_{s},$ $s=0,1,2.$ Let $D^{\ast}=E^{\mid\mathcal{\tau}\mid}\left(  \vec
{z}_{0}e_{0}+\vec{z}_{1}e_{1}+\vec{z}_{2}e_{2}\right)  \oplus E^{\mid
\mathcal{\tau}\mid}\left(  \vec{w}_{0}e_{0}+\vec{w}_{1}e_{1}+\vec{w}_{2}%
e_{2}\right)  $ be a rank two $\varphi$-stable submodule. Assume that
$z_{0}\left(  i\right)  \neq0.$ If $w_{0}\left(  i\right)  \neq0,$ there is no
loss to assume that $z_{0}(i)=w_{0}\left(  i\right)  =1.$ Since $D^{\ast}$ is
$\varphi$-stable, applying $\varphi^{f}$ we see that there exist $\lambda
_{j}(i),\mu_{j}(i)\in E$ such that
\begin{align}
\alpha e_{0}+\beta z_{1}(i)e_{1}+\gamma z_{2}(i)e_{2}  &  =\lambda
_{1}(i)\left(  e_{0}+z_{1}(i)e_{1}+z_{2}(i)e_{2}\right)  +\mu_{1}(i)\left(
e_{0}+w_{1}(i)e_{1}+w_{2}(i)e_{2}\right)  ,\label{2}\\
\alpha e_{0}+\beta w_{1}(i)e_{1}+\gamma w_{2}(i)e_{2}  &  =\lambda
_{2}(i)\left(  e_{0}+z_{1}(i)e_{1}+z_{2}(i)e_{2}\right)  +\mu_{2}(i)\left(
e_{0}+w_{1}(i)e_{1}+w_{2}(i)e_{2}\right)  . \label{3}%
\end{align}
These imply the following:
\begin{align}
&  \left(  \lambda_{1}(i)-\beta\right)  z_{1}(i)+\mu_{1}(i)w_{1}%
(i)=0,\label{4}\\
&  \left(  \alpha-\mu_{2}(i)\right)  z_{1}(i)+\left(  \mu_{2}(i)-\beta\right)
w_{1}(i)=0,\label{5}\\
&  \left(  \lambda_{1}(i)-\gamma\right)  z_{2}(i)+\mu_{1}(i)w_{2}%
(i)=0,\label{6}\\
&  \left(  \alpha-\mu_{2}(i)\right)  z_{2}(i)+\left(  \mu_{2}(i)-\gamma
\right)  w_{2}(i)=0,\label{7}\\
&  \alpha=\lambda_{1}(i)+\mu_{1}(i)=\lambda_{2}(i)+\mu_{2}(i).\ \label{8}%
\end{align}
We have $\left(  z_{1}(i),w_{1}(i)\right)  \neq\left(  0,0\right)  $ or
$\left(  z_{2}(i),w_{2}(i)\right)  \neq\left(  0,0\right)  ,$ and in the first
case equations $\left(  \ref{4}\right)  ,$ $\left(  \ref{5}\right)  $
and\ $\left(  \ref{8}\right)  \ $imply that
\[
\det\left(
\begin{array}
[c]{cc}%
\lambda_{1}(i)-\beta & \mu_{1}(i)\\
\alpha-\mu_{2}(i) & \mu_{2}(i)-\beta
\end{array}
\right)  =\det\left(
\begin{array}
[c]{cc}%
\lambda_{1}(i)-\beta+\mu_{1}(i) & \mu_{1}(i)\\
\alpha-\beta & \mu_{2}(i)-\beta
\end{array}
\right)  =\left(  \alpha-\beta\right)  \left(  \mu_{2}(i)-\mu_{1}%
(i)-\beta\right)  =0.
\]
Hence $\mu_{1}(i)-\mu_{2}(i)=b=\lambda_{2}(i)-\lambda_{1}(i).$ Subtracting
equations $\left(  \ref{2}\right)  $ and $\left(  \ref{3}\right)  $ we get$\ $%
\[
\beta\left(  z_{1}(i)-w_{1}(i)\right)  e_{1}+\gamma\left(  z_{2}%
(i)-w_{2}(i)\right)  e_{2}=\beta\left(  z_{1}(i)-w_{1}(i)\right)  e_{1}%
+\beta\left(  z_{2}(i)-w_{2}(i)\right)  e_{2}.
\]
Since $\beta\neq\gamma,$ the latter implies that $z_{2}(i)=w_{2}(i).$ Then
$z_{1}(i)\neq w_{1}(i)$ and
\begin{align*}
&  e_{\tau_{i}}D^{\ast}=e_{\tau_{i}}E^{\mid\mathcal{\tau}\mid}\left(
e_{0}+z_{1}(i)e_{1}+z_{2}(i)e_{2}\right)  \oplus e_{\tau_{i}}E^{\mid
\mathcal{\tau}\mid}\left(  e_{0}+w_{1}(i)e_{1}+z_{2}(i)e_{2}\right)  =\\
&  e_{\tau_{i}}E^{\mid\mathcal{\tau}\mid}\left(  e_{0}+z_{1}(i)e_{1}%
+z_{2}(i)e_{2}\right)  \oplus e_{\tau_{i}}E^{\mid\mathcal{\tau}\mid}%
e_{1}=e_{\tau_{i}}E^{\mid\mathcal{\tau}\mid}\left(  e_{0}+z_{2}(i)e_{2}%
\right)  \oplus e_{\tau_{i}}E^{\mid\mathcal{\tau}\mid}e_{1}.
\end{align*}
Since the eigenvalues of $\varphi^{f}$ are distinct, the latter is easily seen
to be $\varphi$-stable if and only if $z_{2}(i)=w_{2}(i)=0.$ Arguing similarly
for the remaining cases we see that the only rank two $\varphi$-stable
submodules of $D$ are the $D_{01}:=E^{\mid\mathcal{\tau}\mid}e_{0}\oplus
E^{\mid\mathcal{\tau}\mid}e_{1},D_{02}:=E^{\mid\mathcal{\tau}\mid}e_{0}\oplus
E^{\mid\mathcal{\tau}\mid}e_{2},D_{12}:=E^{\mid\mathcal{\tau}\mid}e_{1}\oplus
E^{\mid\mathcal{\tau}\mid}e_{2}.$ For any $E^{\mid\mathcal{\tau}\mid}$
subspace $D^{\ast}$ of $D$ we have%
\[
t_{H}^{E}\left(  D^{\ast}\right)  =%
%TCIMACRO{\tsum \limits_{j\in\mathbb{Z}}}%
%BeginExpansion
{\textstyle\sum\limits_{j\in\mathbb{Z}}}
%EndExpansion
\dim_{E}\left(  \mathrm{Fil}^{j}\left(  D^{\ast}\right)  /\mathrm{Fil}%
^{j+1}\left(  D^{\ast}\right)  \right)  =%
%TCIMACRO{\tsum \limits_{i=0}^{f-1}}%
%BeginExpansion
{\textstyle\sum\limits_{i=0}^{f-1}}
%EndExpansion%
%TCIMACRO{\tsum \limits_{j\in\mathbb{Z}}}%
%BeginExpansion
{\textstyle\sum\limits_{j\in\mathbb{Z}}}
%EndExpansion
\dim_{E}\left(  \mathrm{Fil}^{j}\left(  e_{\tau_{i}}D^{\ast}\right)
/\mathrm{Fil}^{j+1}\left(  e_{\tau_{i}}D^{\ast}\right)  \right)  .
\]

$\left(  1\right)  $ Assume that $\mathrm{HT}_{\tau_{i}}\left(  e_{\tau_{i}%
}D\right)  =\{0=k_{0}(i)<k_{1}(i)<k_{2}(i)\}.$ Then the filtration of
$e_{\tau_{i}}D$ is given by formula $\left(  \text{\ref{1}}\right)  $. Since
$\mathrm{Fil}^{\mathit{j}}\left(  e_{\tau_{i}}D_{0}\right)  =e_{\tau_{i}}%
D_{0}\cap\mathrm{Fil}^{\mathit{j}}\left(  e_{\tau_{i}}D\right)  ,$ we have%
\[
\ \ \ \ \ \ \ \ \ \ \ \ \mathrm{Fil}^{\mathit{j}}\left(  e_{\tau_{i}}%
D_{0}\right)  =\left\{
\begin{array}
[c]{l}%
e_{\tau_{i}}E^{\mid\mathcal{\tau}\mid}e_{0}\ \mathrm{if}\text{ }j\leq0,\\
\\
\left\{
\begin{array}
[c]{l}%
e_{\tau_{i}}E^{\mid\mathcal{\tau}\mid}e_{0}\ \text{if\ }x_{2}(i)=0,\\
0\ \text{if\ }x_{2}(i)=1,
\end{array}
\right\}  \ \ \mathrm{if}\text{ }1\leq j\leq k_{1}(i),\\
\\
\left\{
\begin{array}
[c]{l}%
e_{\tau_{i}}E^{\mid\mathcal{\tau}\mid}e_{0}\ \text{if }x_{1}(i)=x_{2}%
^{\prime\prime}(i)=0\\
0\ \text{if }x_{1}(i)\neq0\ \text{or }x_{2}^{\prime\prime}(i)\neq0,
\end{array}
\right\}  \ \mathrm{if}\text{ }1+k_{1}\leq j\leq k_{2},\\
\\
0\ \ \mathrm{if}\text{ }j\geq1+k_{2},
\end{array}
\right.
\]
hence%
\[
\ \ \ \ \ \ \ \ \ \ \ \ \ \ \ \ \ \ \ \ \ \ \ t_{H}^{E}\left(  e_{\tau_{i}%
}D_{0}\right)  =\left\{
\begin{array}
[c]{l}%
0\ \text{if\ }x_{2}(i)\neq0,\ \\
k_{1}(i)\ \text{if }x_{2}(i)=0\ \text{and }x_{1}(i)\not =0,\\
k_{2}(i)\ \text{if }x_{1}(i)=x_{2}(i)=0.
\end{array}
\right.
\]
Similarly,%
\[
\mathrm{Fil}^{\mathit{j}}\left(  e_{\tau_{i}}D_{1}\right)  =\left\{
\begin{array}
[c]{l}%
e_{\tau_{i}}E^{\mid\mathcal{\tau}\mid}e_{1}\ \mathrm{if}\text{ }j\leq0,\\
\\
\left\{
\begin{array}
[c]{l}%
e_{\tau_{i}}E^{\mid\mathcal{\tau}\mid}e_{1}\ \text{if\ }x_{2}^{\prime}(i)=0,\\
0\ \text{if\ }x_{2}^{\prime}(i)=1,
\end{array}
\right\}  \ \ \mathrm{if}\text{ }1\leq j\leq k_{1}(i),\\
\\
0\ \mathrm{if}\text{ }1+k_{1}\leq j,
\end{array}
\right.
\]
hence%
\[
\ \ \ \ \ t_{H}^{E}\left(  e_{\tau_{i}}D_{1}\right)  =\left\{
\begin{array}
[c]{l}%
0\ \text{if\ }x_{2}^{\prime}(i)=1,\ \\
k_{1}(i)\ \text{if }x_{2}^{\prime}(i)=0,
\end{array}
\right.
\]
and%
\[
\ \ \ \ \mathrm{Fil}^{\mathit{j}}\left(  e_{\tau_{i}}D_{2}\right)  =\left\{
\begin{array}
[c]{l}%
e_{\tau_{i}}E^{\mid\mathcal{\tau}\mid}e_{2}\ \mathrm{if}\text{ }j\leq0,\\
0\ \ \mathrm{if}\text{ }1\leq j,
\end{array}
\right.
\]
hence%
\[
t_{H}^{E}\left(  D_{2}\right)  =0,
\]%
\[
\ \ \ \ \ \ \ \ \ \ \ \ \ \ \ \ \ \ \ \ \ \ \ \ \ \ \ \ \ \ \ \ \mathrm{Fil}%
^{\mathit{j}}\left(  e_{\tau_{i}}D_{01}\right)  =\left\{
\begin{array}
[c]{l}%
e_{\tau_{i}}E^{\mid\mathcal{\tau}\mid}e_{0}\oplus e_{\tau_{i}}E^{\mid
\mathcal{\tau}\mid}e_{1}\ \mathrm{if}\text{ }j\leq0,\\
\\
\left\{
\begin{array}
[c]{l}%
e_{\tau_{i}}E^{\mid\mathcal{\tau}\mid}\left(  e_{0}-x_{2}(i)e_{1}\right)
\ \text{if\ }x_{2}^{\prime}(i)=1,\\
e_{\tau_{i}}E^{\mid\mathcal{\tau}\mid}e_{0}\oplus e_{\tau_{i}}E^{\mid
\mathcal{\tau}\mid}e_{1}\ \text{if\ }x_{2}(i)=x_{2}^{\prime}(i)=0,\\
e_{\tau_{i}}E^{\mid\mathcal{\tau}\mid}e_{1}\ \text{if\ }x_{2}^{\prime
}(i)=0\ \text{and }x_{2}(i)\neq0,
\end{array}
\right\}  \ \ \mathrm{if}\text{ }1\leq j\leq k_{1}(i),\\
\\
\left\{
\begin{array}
[c]{l}%
e_{\tau_{i}}E^{\mid\mathcal{\tau}\mid}\left(  e_{0}+x_{1}(i)e_{1}\right)
\ \text{if }x_{2}(i)+x_{1}x_{2}^{\prime}(i)=0,\\
0\ \text{if }x_{2}(i)+x_{1}(i)x_{2}^{\prime}(i)\not =0,
\end{array}
\right\}  \ \mathrm{if}\text{ }1+k_{1}\leq j\leq k_{2}(i),\\
\\
0\ \ \mathrm{if}\text{ }j\geq1+k_{2}(i),
\end{array}
\right.
\]
hence%
\[
\ \ \ t_{H}^{E}\left(  e_{\tau_{i}}D_{01}\right)  =\left\{
\begin{array}
[c]{l}%
k_{1}(i)\ \text{if }x_{2}(i)+x_{1}(i)x_{2}^{\prime}(i)\not =0,\\
k_{2}(i)\ \text{if }x_{2}(i)+x_{1}(i)x_{2}^{\prime}(i)=0\ \text{and\ }%
x_{2}^{\prime}(i)=1,\\
k_{1}(i)+k_{2}(i)\ \text{if }x_{2}(i)=x_{2}^{\prime}(i)=0,
\end{array}
\right.
\]

\[
\ \ \ \ \ \ \ \ \ \ \ \ \ \ \ \ \ \ \ \ \ \ \ \ \ \ \ \ \ \ \ \ \ \ \mathrm{Fil}%
^{\mathit{j}}\left(  e_{\tau_{i}}D_{02}\right)  =\left\{
\begin{array}
[c]{l}%
e_{\tau_{i}}E^{\mid\mathcal{\tau}\mid}e_{0}\oplus e_{\tau_{i}}E^{\mid
\mathcal{\tau}\mid}e_{2}\ \mathrm{if}\text{ }j\leq0,\\
e_{\tau_{i}}E^{\mid\mathcal{\tau}\mid}\left(  e_{0}+x_{2}(i)e_{2}\right)
\ \ \mathrm{if}\text{ }1\leq j\leq k_{1}(i),\\
\\
\left\{
\begin{array}
[c]{l}%
e_{\tau_{i}}E^{\mid\mathcal{\tau}\mid}\left(  e_{0}+x_{2}(i)e_{2}\right)
\ \text{if }x_{1}(i)=0,\\
0\ \text{if }x_{1}(i)\neq0,
\end{array}
\right\}  \ \mathrm{if}\text{ }1+k_{1}(i)\leq j\leq k_{2}(i),\\
\\
0\ \ \mathrm{if}\text{ }j\geq1+k_{2}(i),
\end{array}
\right.
\]
hence%
\[
t_{H}^{E}\left(  e_{\tau_{i}}D_{02}\right)  =\left\{
\begin{array}
[c]{l}%
k_{1}(i)\ \text{if }x_{1}(i)\not =0,\\
k_{2}(i)\ \text{if }x_{1}(i)=0.
\end{array}
\right.
\]
Finally,%
\[
\mathrm{Fil}^{\mathit{j}}\left(  e_{\tau_{i}}D_{12}\right)  =\left\{
\begin{array}
[c]{l}%
e_{\tau_{i}}E^{\mid\mathcal{\tau}\mid}e_{1}\oplus e_{\tau_{i}}E^{\mid
\mathcal{\tau}\mid}e_{2}\ \mathrm{if}\text{ }j\leq0,\\
e_{\tau_{i}}E^{\mid\mathcal{\tau}\mid}\left(  e_{1}+x_{2}^{\prime}%
(i)e_{2}\right)  \ \mathrm{if}\text{ }1\leq j\leq k_{1}(i),\\
0\ \ \mathrm{if}\text{ }j\geq1+k_{1}(i).
\end{array}
\right.
\]
Hence%
\[
t_{H}^{E}\left(  e_{\tau_{i}}D_{12}\right)  =k_{1}(i).
\]

$\left(  2\right)  $ Assume that the filtration of $e_{\tau_{i}}D$ is given by
formula (\ref{fil1}). We have%
\[
\mathrm{Fil}^{\mathit{j}}\left(  e_{\tau_{i}}D_{0}\right)  =\left\{
\begin{array}
[c]{l}%
e_{\tau_{i}}D_{0}\ \mathrm{if}\text{ }j\leq0,\\
\\
\left\{
\begin{array}
[c]{l}%
e_{\tau_{i}}D_{0}\ \text{if\ }x_{2}(i)=0,\\
0\ \text{if\ }x_{2}(i)=1,
\end{array}
\right\}  \ \ \mathrm{if}\text{ }1\leq j\leq k(i),\\
\\
0\ \ \mathrm{if}\text{ }j\geq1+k(i),
\end{array}
\right.
\]
hence%
\[
t_{H}^{E}\left(  e_{\tau_{i}}D_{0}\right)  =\left\{
\begin{array}
[c]{l}%
0\ \text{if }x_{2}(i)=1,\\
k(i)\ \text{if }x_{1}(i)=0.
\end{array}
\right.
\]%
\[
\mathrm{Fil}^{\mathit{j}}\left(  e_{\tau_{i}}D_{1}\right)  =\left\{
\begin{array}
[c]{l}%
e_{\tau_{i}}D_{1}\ \mathrm{if}\text{ }j\leq0,\\
\\
\left\{
\begin{array}
[c]{l}%
e_{\tau_{i}}D_{1}\ \text{if\ }x_{2}^{\prime}(i)=0,\\
0\ \text{if\ }x_{2}^{\prime}(i)=1,
\end{array}
\right\}  \ \ \mathrm{if}\text{ }1\leq j\leq k(i),\\
\\
0\ \ \mathrm{if}\text{ }j\geq1+k(i),
\end{array}
\right.
\]
hence%
\[
\ \ \ t_{H}^{E}\left(  e_{\tau_{i}}D_{1}\right)  =\left\{
\begin{array}
[c]{l}%
k(i)\ \text{if }x_{2}^{\prime}(i)=0,\\
0\ \text{if }x_{2}^{\prime}(i)=1.
\end{array}
\right.
\]%
\[
\mathrm{Fil}^{\mathit{j}}\left(  e_{\tau_{i}}D_{2}\right)  =\left\{
\begin{array}
[c]{l}%
e_{\tau_{i}}D_{1}\ \mathrm{if}\text{ }j\leq0,\\
0\ \ \mathrm{if}\text{ }j\geq1,
\end{array}
\right.
\]
hence%
\[
t_{H}^{E}\left(  e_{\tau_{i}}D_{2}\right)  =0.
\]%
\[
\ \ \ \ \ \ \ \ \ \ \ \mathrm{Fil}^{\mathit{j}}\left(  e_{\tau_{i}}%
D_{01}\right)  =\left\{
\begin{array}
[c]{l}%
e_{\tau_{i}}D_{01}\ \mathrm{if}\text{ }j\leq0,\\
\\
\left\{
\begin{array}
[c]{l}%
e_{\tau_{i}}E^{\mid\mathcal{\tau}\mid}\left(  e_{0}+x_{2}(i)e_{1}\right)
\ \text{if\ }x_{2}^{\prime}(i)=1,\\
e_{\tau_{i}}D_{01}\ \text{if\ }x_{2}(i)=x_{2}^{\prime}(i)=0,\\
e_{\tau_{i}}E^{\mid\mathcal{\tau}\mid}e_{1}\ \text{if\ }x_{2}^{\prime
}(i)=0\ \text{and }x_{2}(i)\neq0,
\end{array}
\right\}  \ \ \mathrm{if}\text{ }1\leq j\leq k(i),\\
\\
0\ \ \mathrm{if}\text{ }j\geq1+k(i),
\end{array}
\right.
\]
hence%
\[
\ \ \ \ \ \ \ \ \ t_{H}^{E}\left(  e_{\tau_{i}}D_{01}\right)  =\left\{
\begin{array}
[c]{l}%
k(i)\ \ \text{if }x_{2}^{\prime}(i)=1\ \text{or }x_{2}^{\prime}%
(i)=0\ \text{and }x_{2}(i)\neq0,\\
2k(i)\ \text{if }x_{2}(i)=x_{2}^{\prime}(i)=0.
\end{array}
\right.  \bigskip
\]%
\[
\mathrm{Fil}^{\mathit{j}}\left(  e_{\tau_{i}}D_{02}\right)  =\left\{
\begin{array}
[c]{l}%
e_{\tau_{i}}D_{02}\ \mathrm{if}\text{ }j\leq0,\\
e_{\tau_{i}}E^{\mid\mathcal{\tau}\mid}\left(  e_{0}+x_{2}(i)e_{2}\right)
\ \ \mathrm{if}\text{ }1\leq j\leq k(i),\\
0\ \ \mathrm{if}\text{ }j\geq1+k(i),
\end{array}
\right.
\]
hence%
\[
t_{H}^{E}\left(  e_{\tau_{i}}D_{02}\right)  =k(i).
\]%
\[
\mathrm{Fil}^{\mathit{j}}\left(  e_{\tau_{i}}D_{12}\right)  =\left\{
\begin{array}
[c]{l}%
e_{\tau_{i}}D_{12}\ \mathrm{if}\text{ }j\leq0,\\
e_{\tau_{i}}E^{\mid\mathcal{\tau}\mid}\left(  e_{1}+x_{2}^{\prime}%
(i)e_{2}\right)  \ \mathrm{if}\text{ }1\leq j\leq k(i),\\
0\ \ \mathrm{if}\text{ }j\geq1+k(i),
\end{array}
\right.
\]
hence%
\[
t_{H}^{E}\left(  e_{\tau_{i}}D_{12}\right)  =k(i).
\]
$\left(  3\right)  $ Assume that the filtration of $e_{\tau_{i}}D$ is given by
formula (\ref{filtr'}). We have%
\[
\mathrm{Fil}^{\mathit{j}}\left(  e_{\tau_{i}}D_{0}\right)  =\left\{
\begin{array}
[c]{l}%
e_{\tau_{i}}E^{\mid\mathcal{\tau}\mid}e_{0}\ \mathrm{if}\text{ }j\leq0,\\
\\
\left\{
\begin{array}
[c]{l}%
e_{\tau_{i}}E^{\mid\mathcal{\tau}\mid}e_{0}\ \text{if }x_{1}(i)=x_{2}%
^{\prime\prime}(i)=0\\
0\ \text{if }x_{1}(i)\neq0\ \text{or }x_{2}^{\prime\prime}(i)=1,
\end{array}
\right\}  \ \mathrm{if}\text{ }1\leq j\leq k(i),\\
\\
0\ \ \mathrm{if}\text{ }j\geq1+k(i),
\end{array}
\right.
\]
hence%
\[
\ \ \ \ \ \ \ \ \ \ \ \ \ \ t_{H}^{E}\left(  e_{\tau_{i}}D_{0}\right)
=\left\{
\begin{array}
[c]{l}%
0\ \text{if\ }x_{1}(i)=1\ \text{or }x_{2}^{\prime\prime}(i)=1,\\
k(i)\ \text{if }x_{1}(i)=x_{2}^{\prime\prime}(i)=0.
\end{array}
\right.
\]%
\[
\mathrm{Fil}^{\mathit{j}}\left(  e_{\tau_{i}}D_{1}\right)  =\left\{
\begin{array}
[c]{l}%
e_{\tau_{i}}E^{\mid\mathcal{\tau}\mid}e_{1}\ \mathrm{if}\text{ }j\leq0,\\
0\ \ \mathrm{if}\text{ }1\leq j,
\end{array}
\right.
\]
and
\[
\mathrm{Fil}^{\mathit{j}}\left(  e_{\tau_{i}}D_{2}\right)  =\left\{
\begin{array}
[c]{l}%
e_{\tau_{i}}E^{\mid\mathcal{\tau}\mid}e_{2}\ \mathrm{if}\text{ }j\leq0,\\
0\ \ \mathrm{if}\text{ }1\leq j,
\end{array}
\right.
\]
hence%
\[
t_{H}^{E}\left(  e_{\tau_{i}}D_{1}\right)  =t_{H}^{E}\left(  e_{\tau_{i}}%
D_{2}\right)  =0.
\]%
\[
\ \ \ \ \ \mathrm{Fil}^{\mathit{j}}\left(  e_{\tau_{i}}D_{01}\right)
=\left\{
\begin{array}
[c]{l}%
e_{\tau_{i}}D_{01}\ \mathrm{if}\text{ }j\leq0,\\
\\
\left\{
\begin{array}
[c]{l}%
e_{\tau_{i}}E^{\mid\mathcal{\tau}\mid}\left(  e_{0}+x_{1}(i)e_{1}\right)
\ \text{if }x_{2}^{\prime\prime}(i)=0,\\
0\ \text{if }x_{2}^{\prime\prime}(i)=1,
\end{array}
\right\}  \ \mathrm{if}\text{ }1\leq j\leq k(i),\\
\\
0\ \ \mathrm{if}\text{ }j\geq1+k(i),
\end{array}
\right.
\]
hence%
\[
\ t_{H}^{E}\left(  e_{\tau_{i}}D_{01}\right)  =\left\{
\begin{array}
[c]{l}%
0\ \text{if\ }x_{2}^{\prime\prime}(i)=1,\\
k(i)\ \text{if }x_{2}^{\prime\prime}(i)=0.
\end{array}
\right.
\]%
\[
\mathrm{Fil}^{\mathit{j}}\left(  e_{\tau_{i}}D_{02}\right)  =\left\{
\begin{array}
[c]{l}%
e_{\tau_{i}}D_{02}\ \mathrm{if}\text{ }j\leq0,\\
\\
\left\{
\begin{array}
[c]{l}%
e_{\tau_{i}}E^{\mid\mathcal{\tau}\mid}\left(  e_{0}+x_{2}(i)e_{2}\right)
\ \text{if }x_{1}(i)=0,\\
0\ \text{if }x_{1}(i)=1,
\end{array}
\right\}  \ \mathrm{if}\text{ }1\leq j\leq k(i),\\
\\
0\ \ \mathrm{if}\text{ }j\geq1+k(i),
\end{array}
\right.
\]
hence%
\[
\ \ \ t_{H}^{E}\left(  e_{\tau_{i}}D_{02}\right)  =\left\{
\begin{array}
[c]{l}%
0\ \text{if\ }x_{1}(i)=1,\\
k(i)\ \text{if }x_{1}(i)=0.
\end{array}
\right.
\]%
\[
\mathrm{Fil}^{\mathit{j}}\left(  e_{\tau_{i}}D_{12}\right)  =\left\{
\begin{array}
[c]{l}%
e_{\tau_{i}}D_{12}\ \mathrm{if}\text{ }j\leq0,\\
0\ \ \mathrm{if}\text{ }1\leq j,
\end{array}
\right.
\]%
\[
t_{H}^{E}\left(  e_{\tau_{i}}D_{12}\right)  =0.
\]
$\left(  4\right)  $ If the filtration of $e_{\tau_{i}}D$ is given by formula
(\ref{triv fil}), then all the Hodge invariants are $0.$ For weak
admissibility, we must have $t_{N}^{E}\left(  D\right)  =t_{H}^{E}\left(
D\right)  $ and $t_{N}^{E}\left(  D^{\ast}\right)  =t_{H}^{E}\left(  D^{\ast
}\right)  $ for any $\varphi$-stable subspace $D^{\ast}$ of $D.$ Clearly%
\[
t_{N}^{E}\left(  D\right)  =v_{p}\left(  \mathrm{Nm}_{\varphi}\left(  \vec
{a}\right)  \right)  +v_{p}(\mathrm{Nm}_{\varphi}(\vec{b}))+v_{p}\left(
\mathrm{Nm}_{\varphi}\left(  \vec{c}\right)  \right)
\]
and%
\[
t_{H}^{E}\left(  D\right)  =\sum\limits_{i\in I_{1}}\left(  k_{1}%
(i)+k_{2}(i)\right)  +\sum\limits_{i\in I_{2}}2k(i)+\sum\limits_{i\in I_{3}%
}k(i).
\]
Also,$\ t_{N}^{E}\left(  D_{0}\right)  =v_{p}\left(  \mathrm{Nm}_{\varphi
}\left(  \vec{a}\right)  \right)  $ and
\begin{align*}
t_{H}^{E}\left(  D_{0}\right)  =  &  \sum\limits_{i\in I_{1}}\left\{
\begin{array}
[c]{l}%
0\ \text{if\ }x_{2}(i)\neq0,\ \\
k_{1}(i)\ \text{if }x_{2}(i)=0\ \text{and }x_{1}(i)\not =0,\\
k_{2}(i)\ \text{if }x_{1}(i)=x_{2}(i)=0
\end{array}
\right\}  +\sum\limits_{i\in I_{2}}\left\{
\begin{array}
[c]{l}%
0\ \text{if }x_{2}(i)=1,\\
k(i)\ \text{if }x_{1}(i)=0
\end{array}
\right\}  +\\
&  \ \ \ \ \ \text{\ }\\
+  &  \sum\limits_{i\in I}\left\{
\begin{array}
[c]{l}%
0\ \text{if\ }x_{1}(i)=1\ \text{or }x_{2}^{\prime\prime}(i)=1,\\
k(i)\ \text{if }x_{1}(i)=x_{2}^{\prime\prime}(i)=0
\end{array}
\right\}  .
\end{align*}
Similarly, we must have%
\begin{align*}
&  v_{p}\left(  \mathrm{Nm}_{\varphi}(\vec{b})\right)  \geq\sum\limits_{i\in
I_{1}}\left\{
\begin{array}
[c]{l}%
0\ \text{if\ }x_{2}^{\prime}(i)=1,\ \\
k_{1}(i)\ \text{if }x_{2}^{\prime}(i)=0
\end{array}
\right\}  +\sum\limits_{i\in I_{2}}\left\{
\begin{array}
[c]{l}%
k(i)\ \text{if }x_{2}^{\prime}(i)=0,\\
0\ \text{if }x_{2}^{\prime}(i)=1
\end{array}
\right\}  ,\\
& \\
&  v_{p}\left(  \mathrm{Nm}_{\varphi}(\vec{c})\right)  \geq0,\\
& \\
&  v_{p}\left(  \mathrm{Nm}_{\varphi}\left(  \vec{a}\right)  \right)
+v_{p}\left(  \mathrm{Nm}_{\varphi}(\vec{b})\right)  \geq\sum\limits_{i\in
J_{1}}\left\{
\begin{array}
[c]{l}%
k_{1}(i)\ \text{if }x_{2}(i)+x_{1}(i)x_{2}^{\prime}(i)\not =0,\\
k_{2}(i)\ \text{if }x_{2}(i)+x_{1}(i)x_{2}^{\prime}(i)=0\ \text{and\ }%
x_{2}^{\prime}(i)=1,\\
k_{1}(i)+k_{2}(i)\ \text{if }x_{2}(i)=x_{2}^{\prime}(i)=0
\end{array}
\right\}  +\\
& \\
&  +\sum\limits_{i\in J_{2}}\left\{
\begin{array}
[c]{l}%
k(i)\ \ \text{if }x_{2}^{\prime}(i)=1\ \text{or }x_{2}^{\prime}%
(i)=0\ \text{and }x_{2}(i)\neq0,\\
2k(i)\ \text{if }x_{2}(i)=x_{2}^{\prime}(i)=0
\end{array}
\right\}  +\sum\limits_{i\in J_{3}}\left\{
\begin{array}
[c]{l}%
0\ \text{if\ }x_{2}^{\prime\prime}(i)=1,\\
k(i)\ \text{if }x_{2}^{\prime\prime}(i)=0
\end{array}
\right\}  ,\\
& \\
&  v_{p}\left(  \mathrm{Nm}_{\varphi}\left(  \vec{a}\right)  \right)
+v_{p}\left(  \mathrm{Nm}_{\varphi}(\vec{c})\right)  \geq\sum\limits_{i\in
J_{1}}\left\{
\begin{array}
[c]{l}%
k_{1}(i)\ \text{if }x_{1}(i)\not =0,\\
k_{2}(i)\ \text{if }x_{1}(i)=0
\end{array}
\right\}  +\sum\limits_{i\in J_{2}}k(i)+\sum\limits_{i\in J_{3}}\left\{
\begin{array}
[c]{l}%
0\ \text{if\ }x_{1}(i)=1,\\
k(i)\ \text{if }x_{1}(i)=0
\end{array}
\right\}  ,\\
&  \text{and}\\
&  v_{p}\left(  \mathrm{Nm}_{\varphi}(\vec{b})\right)  +v_{p}\left(
\mathrm{Nm}_{\varphi}(\vec{c})\right)  \geq\sum\limits_{i\in J_{1}}%
k_{1}(i)+\sum\limits_{i\in J_{2}}k(i).
\end{align*}

\end{proof}

\begin{proof}
[Proof of Proposition \ref{iso}]Let $h:\mathcal{D}\left(  \underline
{a},\underline{x}\right)  \rightarrow\mathcal{D}\left(  \underline{a_{1}%
},\underline{y}\right)  $ an $E^{\mid\mathcal{\tau}\mid}$-linear bijection.
The map $h$ is an isomorphism or filtered $\varphi$-modules if and only if
$h\left(  e_{\tau_{i}}\mathrm{Fil}^{j}\mathcal{D}\left(  \underline
{a},\underline{x}\right)  \right)  =e_{\tau_{i}}\mathrm{Fil}^{j}%
\mathcal{D}\left(  \underline{a},\underline{y}\right)  $ for all $i.$ Let
$H:=[h]_{\underline{e}}^{\underline{e}_{1}}=\left(  \vec{h}_{ij}\right)  .$
Since $h\varphi=\varphi_{1}h,$ we have $h\varphi^{f}=\varphi_{1}^{f}h,$ and
since $\varphi^{f}$ and $\varphi_{1}^{f}$ are $E^{\mid\mathcal{\tau}\mid}%
$-linear,\
\[
H\cdot\mathrm{diag}\left(  \mathrm{Nm}_{\varphi}(\vec{a}),\mathrm{Nm}%
_{\varphi}(\vec{b}),\mathrm{Nm}_{\varphi}(\vec{c})\right)  =\mathrm{diag}%
\left(  \mathrm{Nm}_{\varphi}(\vec{a}_{1}),\mathrm{Nm}_{\varphi}(\vec{b}%
_{1}),\mathrm{Nm}_{\varphi}(\vec{c}_{1})\right)  \cdot H,
\]
which implies that%
\begin{align}
&  \left(  \mathrm{Nm}_{\varphi}(\vec{a})-\mathrm{Nm}_{\varphi}(\vec{a}%
_{1})\right)  \vec{h}_{11}=\vec{0},\ \left(  \mathrm{Nm}_{\varphi}(\vec
{a})-\mathrm{Nm}_{\varphi}(\vec{b}_{1})\right)  \vec{h}_{21}=\vec{0},\ \left(
\mathrm{Nm}_{\varphi}(\vec{a})-\mathrm{Nm}_{\varphi}(\vec{c}_{1})\right)
\vec{h}_{31}=\vec{0},\nonumber\\
&  \left(  \mathrm{Nm}_{\varphi}(\vec{b})-\mathrm{Nm}_{\varphi}(\vec{a}%
_{1})\right)  \vec{h}_{12}=\vec{0},~\left(  \mathrm{Nm}_{\varphi}(\vec
{b})-\mathrm{Nm}_{\varphi}(\vec{b}_{1})\right)  \vec{h}_{22}=\vec{0},\ \left(
\mathrm{Nm}_{\varphi}(\vec{b})-\mathrm{Nm}_{\varphi}(\vec{c}_{1})\right)
\vec{h}_{32}=\vec{0},\label{2'}\\
&  \left(  \mathrm{Nm}_{\varphi}(\vec{c})-\mathrm{Nm}_{\varphi}(\vec{a}%
_{1})\right)  \vec{h}_{13}=\vec{0},\ \left(  \mathrm{Nm}_{\varphi}(\vec
{c})-\mathrm{Nm}_{\varphi}(\vec{b}_{1})\right)  \vec{h}_{23}=\vec{0},\ \left(
\mathrm{Nm}_{\varphi}(\vec{c})-\mathrm{Nm}_{\varphi}(\vec{c}_{1})\right)
\vec{h}_{33}=\vec{0}.\nonumber
\end{align}
Clearly $\{\mathrm{Nm}_{\varphi}(\vec{a}),\mathrm{Nm}_{\varphi}(\vec
{b}),\mathrm{Nm}_{\varphi}(\vec{c})\}=$ $\{\mathrm{Nm}_{\varphi}(\vec{a}%
_{1}),\mathrm{Nm}_{\varphi}(\vec{b}_{1}),\mathrm{Nm}_{\varphi}(\vec{c}_{1})\}$
and we have the following cases:

$\left(  1\right)  \ $If $\mathrm{Nm}_{\varphi}(\vec{a})=\mathrm{Nm}_{\varphi
}(\vec{a}_{1}),$ $\mathrm{Nm}_{\varphi}(\vec{b})=\mathrm{Nm}_{\varphi}(\vec
{b}_{1}),$ $\mathrm{Nm}_{\varphi}(\vec{c})=\mathrm{Nm}_{\varphi}(\vec{c}%
_{1}).$ Since the eigenvalues of frobenius are distinct, equations $\left(
\text{\ref{2'}}\right)  $ are equivalent to $\vec{h}_{ij}=\vec{0}$ for all
$i\neq j.$ We have $\left(  he_{0},he_{1},he_{2}\right)  =\left(  \vec{h}%
_{11}e_{0},\vec{h}_{22}e_{1},\vec{h}_{33}e_{2}\right)  .$

$\left(  a\right)  $ If the filtration of $e_{\tau_{i}}\mathrm{Fil}%
^{j}\mathcal{D}\left(  \underline{a},\underline{x}\right)  $ is given by
formula (\ref{1}). Then $h\left(  e_{\tau_{i}}\mathrm{Fil}^{j}\mathcal{D}%
\left(  \underline{a},\underline{x}\right)  \right)  =e_{\tau_{i}}%
\mathrm{Fil}^{j}\mathcal{D}\left(  \underline{a},\underline{y}\right)  $ is
equivalent to
\begin{align}
&  E\left(  e_{0}+y_{1}(i)e_{1}+y_{2}^{\prime\prime}(i)e_{2}\right)  =E\left(
h_{11}(i)e_{0}+x_{1}(i)h_{22}(i)e_{1}+x_{2}^{\prime\prime}(i)h_{33}%
(i)e_{2}\right)  ,\ \label{4'}\\
&  \text{where }x_{2}^{\prime\prime}(i)=x_{2}(i)+x_{1}(i)x_{2}^{\prime
}(i)\ \text{and }y_{2}^{\prime\prime}(i)=y_{2}(i)+y_{1}(i)y_{2}^{\prime
}(i),\text{\ and}\nonumber\\
&  E\left(  e_{0}+y_{2}(i)e_{2}\right)  \oplus E\left(  e_{1}+y_{2}^{\prime
}(i)e_{2}\right)  =E\left(  h_{11}(i)e_{0}+x_{2}(i)h_{33}(i)e_{2}\right)
\oplus E\left(  h_{22}(i)e_{1}+x_{2}^{\prime}(i)h_{33}(i)e_{2}\right)  .
\label{5'}%
\end{align}
Equation $\left(  \ref{4'}\right)  $ is equivalent to $y_{1}(i)h_{11}%
(i)e_{1}+y_{2}^{\prime\prime}(i)h_{11}(i)e_{2}=x_{1}(i)h_{22}(i)e_{1}%
+x_{2}^{\prime\prime}(i)h_{33}(i)e_{2}$ and the latter equivalent to%
\[
y_{1}(i)h_{11}(i)=x_{1}(i)h_{22}(i)\ \text{and\ }y_{2}^{\prime\prime}%
(i)h_{11}(i)=x_{2}^{\prime\prime}(i)h_{33}(i).
\]
Equation $\left(  \ref{5'}\right)  $ is equivalent to
\begin{align*}
h_{11}(i)e_{0}+x_{2}(i)h_{33}(i)e_{2}  &  =\lambda_{1}\left(  e_{0}%
+y_{2}(i)e_{2}\right)  ,\\
\left(  h_{22}(i)e_{1}+x_{2}^{\prime}(i)h_{33}(i)e_{2}\right)   &
=\lambda_{2}\left(  e_{1}+y_{2}^{\prime}(i)e_{2}\right)  ,
\end{align*}
for some $\lambda_{1},\lambda_{2}\in E.$ Hence $\lambda_{1}=h_{11}(i),$
$\lambda_{2}=h_{22}(i),$ $h_{11}(i)y_{2}(i)=x_{2}(i)h_{33}(i)\ $%
and\ $h_{22}(i)y_{2}^{\prime}(i)=x_{2}^{\prime}(i)h_{33}(i).$ We have%
\begin{align}
h_{11}(i)y_{2}(i)  &  =x_{2}(i)h_{33}(i)\ \text{and}\ h_{22}(i)y_{2}^{\prime
}(i)=x_{2}^{\prime}(i)h_{33}(i),\label{8'}\\
y_{1}(i)h_{11}(i)  &  =x_{1}(i)h_{22}(i)\ \text{and\ }y_{2}^{\prime\prime
}(i)h_{11}(i)=x_{2}^{\prime\prime}(i)h_{33}(i). \label{9'}%
\end{align}

\noindent Since $x_{2}^{\prime}(i),y_{2}^{\prime}(i)\in\{0,1\},$
$h_{22}(i)y_{2}^{\prime}(i)=x_{2}^{\prime}(i)h_{33}(i)$ implies that
$x_{2}^{\prime}(i)=y_{2}^{\prime}(i)$ and $h_{22}(i)=h_{33}(i)$ if
$x_{2}^{\prime}(i)=y_{2}^{\prime}(i)=1.$ We have$\ h_{11}(i)y_{2}%
(i)=x_{2}(i)h_{33}(i)\ ,\ h_{22}(i)y_{2}^{\prime}(i)=x_{2}^{\prime}%
(i)h_{33}(i),\ y_{1}(i)h_{11}(i)=x_{1}(i)h_{22}(i),\ $and$\ y_{2}%
^{\prime\prime}(i)h_{11}(i)=x_{2}^{\prime\prime}(i)h_{33}(i).$

$\left(  i\right)  $ If $x_{2}^{\prime}(i)=y_{2}^{\prime}(i)=1.$
Then$\ h_{11}(i)y_{2}(i)=x_{2}(i)h_{33}(i),\ h_{22}(i)=h_{33}(i),\ y_{1}%
(i)h_{11}(i)=x_{1}(i)h_{22}(i),\ $and $\left(  y_{1}(i)+y_{2}(i)\right)
h_{11}(i)=\left(  x_{1}(i)+x_{2}(i)\right)  h_{33}(i),$ or equivalently,
\begin{align*}
y_{2}(i)h_{11}(i)  &  =x_{2}(i)h_{22}(i)\ \\
y_{1}(i)h_{11}(i)  &  =x_{1}(i)h_{22}(i)\\
h_{22}(i)  &  =h_{33}(i).
\end{align*}

$\left(  ia\right)  $ If $x_{2}(i)=1.$ Then $y_{2}(i)=1$ and $h_{11}%
(i)=h_{22}(i),\ y_{1}(i)=x_{1}(i),\ $and $h_{22}(i)=h_{33}(i).$ Then $h\left(
e_{\tau_{i}}\mathrm{Fil}^{j}\mathcal{D}\left(  \underline{a},\underline
{x}\right)  \right)  =e_{\tau_{i}}\mathrm{Fil}^{j}\mathcal{D}\left(
\underline{a},\underline{y}\right)  $ is equivalent to%
\[
x_{2}(i)=x_{2}^{\prime}(i)=y_{2}(i)=y_{2}^{\prime}(i)=1,\ \text{and}%
\ y_{1}(i)=x_{1}(i).
\]

$\left(  ib\right)  $ If $x_{2}(i)=0.$ Then%
\begin{align*}
y_{2}(i)h_{11}(i)  &  =x_{2}(i)h_{22}(i),\\
y_{1}(i)h_{11}(i)  &  =x_{1}(i)h_{22}(i),\\
h_{22}(i)  &  =h_{33}(i),
\end{align*}
hence $x_{2}(i)=y_{2}(i)=0$ and $x_{1}(i)\not =0$ if and only if $y_{1}%
(i)\neq0.$ Then $h\left(  e_{\tau_{i}}\mathrm{Fil}^{j}\mathcal{D}\left(
\underline{a},\underline{x}\right)  \right)  =e_{\tau_{i}}\mathrm{Fil}%
^{j}\mathcal{D}\left(  \underline{a},\underline{y}\right)  $ is equivalent to%
\[
x_{2}^{\prime}(i)=y_{2}^{\prime}(i)=1,\ x_{2}(i)=y_{2}(i)=0,\ \text{and\ }%
x_{1}(i)\not =0\ \text{if and only if\ }y_{1}(i)\neq0.
\]

$\left(  ii\right)  \ $If $x_{2}^{\prime}(i)=y_{2}^{\prime}(i)=0.$%
\begin{align*}
y_{1}(i)h_{11}(i)  &  =x_{1}(i)h_{22}(i)\\
y_{2}(i)h_{11}(i)  &  =x_{2}(i)h_{33}(i).
\end{align*}
Then equations $\left(  \ref{8'}\right)  ,$ $\left(  \ref{9'}\right)  $ are
equivalent to$\ h_{11}(i)y_{2}(i)=x_{2}(i)h_{33}(i)\ $and $y_{1}%
(i)h_{11}(i)=x_{1}(i)h_{22}(i).$ Then $h\left(  e_{\tau_{i}}\mathrm{Fil}%
^{j}\mathcal{D}\left(  \underline{a},\underline{x}\right)  \right)
=e_{\tau_{i}}\mathrm{Fil}^{j}\mathcal{D}\left(  \underline{a},\underline
{y}\right)  $ is equivalent to $x_{2}(i)=1\ $if and only if $y_{2}(i)=1$ and
$x_{1}(i)\neq0\ $if and only if $y_{1}(i)\neq0.$%
\[
x_{2}^{\prime}(i)=y_{2}^{\prime}(i)=0,\ \left(  x_{2}(i)=1\ \text{if and only
if\ }y_{2}(i)=1\right)  \ \text{and\ }\left(  x_{1}(i)\neq0\ \text{if and only
if }y_{1}(i)\neq0\right)  .
\]

Hence $e_{\tau_{i}}\mathcal{D}\left(  \underline{a},\underline{x}\right)
\ $and $e_{\tau_{i}}\mathcal{D}\left(  \underline{a_{1}},\underline{y}\right)
$ are isomorphic if and only if either

$x_{2}^{\prime}(i)=y_{2}^{\prime}(i)=1,\ x_{2}(i)=y_{2}(i)=1,\ $%
and$\ y_{1}(i)=x_{1}(i)$ or

$x_{2}^{\prime}(i)=y_{2}^{\prime}(i)=1,\ x_{2}(i)=y_{2}(i)=0,\ $%
and\ $x_{1}(i)\not =0\ $if and only if\ $y_{1}(i)\neq0,$ or

$x_{2}^{\prime}(i)=y_{2}^{\prime}(i)=0,\ x_{2}(i)=1\ $if and only
if\ $y_{2}(i)=1,\ $and\ $x_{1}(i)\neq0\ $if and only if $y_{1}(i)\neq0.$

$\left(  b\right)  $ If the filtration of $e_{\tau_{i}}\mathrm{Fil}%
^{j}\mathcal{D}\left(  \underline{a},\underline{x}\right)  $ is given by
formula (\ref{fil1}). Then $h\left(  e_{\tau_{i}}\mathrm{Fil}^{j}%
\mathcal{D}\left(  \underline{a},\underline{x}\right)  \right)  =e_{\tau_{i}%
}\mathrm{Fil}^{j}\mathcal{D}\left(  \underline{a},\underline{y}\right)  $ is
equivalent to
\[
E\left(  e_{0}+y_{2}(i)e_{2}\right)  \oplus E\left(  e_{1}+y_{2}^{\prime
}(i)e_{2}\right)  =E\left(  h_{11}(i)e_{0}+x_{2}(i)h_{33}(i)e_{2}\right)
\oplus E\left(  h_{22}(i)e_{1}+x_{2}^{\prime}(i)h_{33}(i)e_{2}\right)
\]
and the latter equivalent to
\begin{align*}
h_{11}(i)e_{0}+x_{2}(i)h_{33}(i)e_{2}  &  =\lambda_{1}\left(  e_{0}%
+y_{2}(i)e_{2}\right)  ,\\
\left(  h_{22}(i)e_{1}+x_{2}^{\prime}(i)h_{33}(i)e_{2}\right)   &
=\lambda_{2}\left(  e_{1}+y_{2}^{\prime}(i)e_{2}\right)  ,
\end{align*}
for some $\lambda_{1},\lambda_{2}\in E.$ Hence $\lambda_{1}=h_{11}(i),$
$\lambda_{1}y_{2}(i)=x_{2}(i)h_{33}(i),$ $\lambda_{2}=h_{22}(i),$ and
$\lambda_{2}y_{2}^{\prime}(i)=x_{2}^{\prime}(i)h_{33}(i).$ These imply that
$y_{2}(i)h_{11}(i)=x_{2}(i)h_{33}(i),$ and $h_{22}(i)y_{2}^{\prime}%
(i)=x_{2}^{\prime}(i)h_{33}(i),$ hence%
\[
h\left(  e_{\tau_{i}}\mathrm{Fil}^{j}\mathcal{D}\left(  \underline
{a},\underline{x}\right)  \right)  =e_{\tau_{i}}\mathrm{Fil}^{j}%
\mathcal{D}\left(  \underline{a},\underline{y}\right)
\]
if and only if $x_{2}(i)\not =0\ $if and only if $y_{2}(i),$ and
$x_{2}^{\prime}(i)\not =0\ $if and only if $y_{2}^{\prime}(i)\neq0.$ Since
$x_{2}(i),y_{2}(i),x_{2}^{\prime}(i),y_{2}^{\prime}(i)\in\{0,1\}$ this is
equivalent to $x_{2}(i)=y_{2}(i)$ and $x_{2}^{\prime}(i)=y_{2}^{\prime}(i).$

$\left(  c\right)  $ If the filtration of $e_{\tau_{i}}\mathrm{Fil}%
^{j}\mathcal{D}\left(  \underline{a},\underline{x}\right)  $ is given by
formula (\ref{filtr'}). Then $h\left(  e_{\tau_{i}}\mathrm{Fil}^{j}%
\mathcal{D}\left(  \underline{a},\underline{x}\right)  \right)  =e_{\tau_{i}%
}\mathrm{Fil}^{j}\mathcal{D}\left(  \underline{a},\underline{y}\right)  $ is
equivalent to%
\[
E\left(  e_{0}+y_{1}(i)e_{1}+y_{2}^{\prime\prime}(i)e_{2}\right)  =E\left(
h_{11}(i)e_{0}+x_{1}(i)h_{22}(i)e_{1}+x_{2}^{\prime\prime}(i)h_{33}%
(i)e_{2}\right)  ,\
\]
which is in turn equivalent to $y_{1}(i)h_{11}(i)=x_{1}(i)h_{22}%
(i)\ $and\ $y_{2}^{\prime\prime}(i)h_{11}(i)=x_{2}^{\prime\prime}%
(i)h_{33}(i).$ Hence $h\left(  e_{\tau_{i}}\mathrm{Fil}^{j}\mathcal{D}\left(
\underline{a},\underline{x}\right)  \right)  =e_{\tau_{i}}\mathrm{Fil}%
^{j}\mathcal{D}\left(  \underline{a},\underline{y}\right)  $ if and only
$x_{1}(i)=y_{1}(i)$ and $x_{2}^{\prime\prime}(i)=y_{2}^{\prime\prime}(i).$

$\left(  2\right)  $ If $\mathrm{Nm}_{\varphi}(\vec{a})=\mathrm{Nm}_{\varphi
}(\vec{b}_{1}),\ \mathrm{Nm}_{\varphi}(\vec{b})=\mathrm{Nm}_{\varphi}(\vec
{a}_{1}),\ \mathrm{Nm}_{\varphi}(\vec{c})=\mathrm{Nm}_{\varphi}(\vec{c}_{1}).$
In this case equations $\left(  \ref{2'}\right)  $ imply that
\[
H=\left(
\begin{array}
[c]{ccc}%
\vec{0} & \vec{h}_{12} & \vec{0}\\
\vec{h}_{21} & \vec{0} & \vec{0}\\
\vec{0} & \vec{0} & \vec{h}_{33}%
\end{array}
\right)
\]
and $\left(  he_{0},he_{1},he_{2}\right)  =\left(
\begin{array}
[c]{ccc}%
\vec{h}_{21}e_{1}, & \vec{h}_{12}e_{0}, & \vec{h}_{33}e_{2}%
\end{array}
\right)  .$

$\left(  a\right)  $ If the filtration of $e_{\tau_{i}}\mathrm{Fil}%
^{j}\mathcal{D}\left(  \underline{a},\underline{x}\right)  $ is given by
formula (\ref{1}). Then $h\left(  e_{\tau_{i}}\mathrm{Fil}^{j}\mathcal{D}%
\left(  \underline{a},\underline{x}\right)  \right)  =e_{\tau_{i}}%
\mathrm{Fil}^{j}\mathcal{D}\left(  \underline{a},\underline{y}\right)  $ is
equivalent to%
\begin{align}
&  E\left(  e_{0}+y_{1}(i)e_{1}+y_{2}^{\prime\prime}(i)e_{2}\right)  =E\left(
x_{1}(i)h_{12}(i)e_{0}+h_{21}(i)e_{1}+x_{2}^{\prime\prime}(i)h_{33}%
(i)e_{2}\right)  ,\ \label{10'}\\
&  \text{where }x_{2}^{\prime\prime}(i)=x_{2}(i)+x_{1}(i)x_{2}^{\prime
}(i)\ \text{and }y_{2}^{\prime\prime}(i)=y_{2}(i)+y_{1}(i)y_{2}^{\prime
}(i),\text{\ and}\nonumber\\
&  E\left(  e_{0}+y_{2}(i)e_{2}\right)  \oplus E\left(  e_{1}+y_{2}^{\prime
}(i)e_{2}\right)  =E\left(  h_{21}(i)e_{1}+x_{2}(i)h_{33}(i)e_{2}\right)
\oplus E\left(  h_{12}(i)e_{0}+x_{2}^{\prime}(i)h_{33}(i)e_{2}\right)  .
\label{11'}%
\end{align}
Equation $\left(  \ref{10'}\right)  $ is equivalent to $h_{21}(i)e_{1}%
+x_{2}^{\prime\prime}(i)h_{33}(i)e_{2}=y_{1}(i)x_{1}(i)h_{12}(i)e_{1}%
+y_{2}^{\prime\prime}(i)x_{1}(i)h_{12}(i)e_{2}$ or equivalently%
\[
h_{21}(i)=y_{1}(i)x_{1}(i)h_{12}(i)\ \text{and\ }x_{2}^{\prime\prime}%
(i)h_{33}(i)=y_{2}^{\prime\prime}(i)x_{1}(i)h_{12}(i).
\]
Equation $\left(  \ref{11'}\right)  $ is equivalent to$\ h_{21}(i)e_{1}%
+x_{2}(i)h_{33}(i)e_{2}=\lambda_{1}\left(  e_{1}+y_{2}^{\prime}(i)e_{2}%
\right)  $ and $h_{12}(i)e_{0}+x_{2}^{\prime}(i)h_{33}(i)e_{2}=\lambda
_{2}\left(  e_{0}+y_{2}(i)e_{2}\right)  ,$ for some $\lambda_{1},\lambda
_{2}\in E.$ So $\lambda_{1}=h_{21}(i),\ \lambda_{2}=h_{12}(i),\ x_{2}%
(i)h_{33}(i)=h_{21}(i)y_{2}^{\prime}(i),\ x_{2}^{\prime}(i)h_{33}%
(i)=h_{12}(i)y(i)_{2},\ $

$h_{21}(i)=y_{1}(i)x_{1}(i)h_{12}(i),\ x_{2}^{\prime\prime}(i)h_{33}%
(i)=y_{2}^{\prime\prime}(i)x_{1}(i)h_{12}(i).$

$\left(  i\right)  $ If $x_{2}^{\prime}(i)=y_{2}^{\prime}(i)=0.$ Then
$x_{2}(i)=y_{2}(i)=0$ and $h_{21}(i)=y_{1}(i)x_{1}(i)h_{12}(i).$ In this case,
$e_{\tau_{i}}\mathcal{D}\left(  \underline{a},\underline{x}\right)  \simeq
e_{\tau_{i}}\mathcal{D}\left(  \underline{a_{1}},\underline{y}\right)  $ if
and only if $x_{2}^{\prime}(i)=y_{2}^{\prime}(i)=\ x_{2}(i)=y_{2}(i)=0,\ $and
$x_{1}(i)y_{1}(i)\neq0.$

$\left(  ii\right)  \ $If $x_{2}^{\prime}(i)=0,\ y_{2}^{\prime}(i)=1.$ Then
$y_{2}(i)=0,\ x_{1}(i)x_{2}(i)y_{1}(i)\neq0,\ h_{21}(i)=x_{2}(i)h_{33}%
(i),\ h_{21}(i)=y_{1}(i)x_{1}(i)h_{12}(i),$ and$\ x_{2}(i)h_{33}%
(i)=h_{21}(i),$ $h_{21}(i)=y_{1}(i)x_{1}(i)h_{12}(i).$ Hence$\ e_{\tau_{i}%
}\mathcal{D}\left(  \underline{a},\underline{x}\right)  \simeq e_{\tau_{i}%
}\mathcal{D}\left(  \underline{a_{1}},\underline{y}\right)  $ if and only if
$x_{2}(i)=x_{2}^{\prime}(i)=y_{2}^{\prime}(i)=1,\ y_{2}=0,\ y_{1}%
(i)x_{1}(i)\neq0.$

$\left(  iii\right)  $ $\ $If $x_{2}^{\prime}(i)=1,\ y_{2}^{\prime}(i)=0.$
Then$\ x_{2}(i)=0,$ $h_{33}(i)=h_{12}(i)y_{2}(i),$ $h_{21}(i)=y_{1}%
(i)x_{1}(i)h_{12}(i),$ and we must have $x_{1}(i)y_{1}(i)\neq0\ $and
$y_{2}(i)=1.$ Hence $e_{\tau_{i}}\mathcal{D}\left(  \underline{a}%
,\underline{x}\right)  \simeq e_{\tau_{i}}\mathcal{D}\left(  \underline{a_{1}%
},\underline{y}\right)  $ if and only if $x_{2}^{\prime}(i)=y_{2}%
(i)=1,\ x_{2}(i)=y_{2}^{\prime}(i)=0,\ x_{1}(i)y_{1}(i)\neq0.$

$\left(  iv\right)  $ If $x_{2}^{\prime}(i)=y_{2}^{\prime}(i)=0,$
then$\ x_{2}(i)=0,\ y_{2}(i)=0,\ $and $h_{21}(i)=y_{1}(i)x_{1}(i)h_{12}(i).$
In this case $e_{\tau_{i}}\mathcal{D}\left(  \underline{a},\underline
{x}\right)  \simeq e_{\tau_{i}}\mathcal{D}\left(  \underline{a_{1}}%
,\underline{y}\right)  $ if%
\[
x_{2}(i)=x_{2}^{\prime}(i)=y_{2}(i)=y_{2}^{\prime}(i)=0,\ \text{and }%
x_{1}(i)y_{1}(i)\neq0.
\]

$\left(  b\right)  $ If the filtration of $e_{\tau_{i}}\mathrm{Fil}%
^{j}\mathcal{D}\left(  \underline{a},\underline{x}\right)  $ is given by
formula (\ref{fil1}). Then $h\left(  e_{\tau_{i}}\mathrm{Fil}^{j}%
\mathcal{D}\left(  \underline{a},\underline{x}\right)  \right)  =e_{\tau_{i}%
}\mathrm{Fil}^{j}\mathcal{D}\left(  \underline{a},\underline{y}\right)  $ is
equivalent to%
\[
E\left(  e_{0}+y_{2}(i)e_{2}\right)  \oplus E\left(  e_{1}+y_{2}^{\prime
}(i)e_{2}\right)  =E\left(  h_{21}(i)e_{1}+x_{2}(i)h_{33}(i)e_{2}\right)
\oplus E\left(  h_{12}(i)e_{0}+x_{2}^{\prime}(i)h_{33}(i)e_{2}\right)
\]
and the latter is equivalent to $h_{21}(i)e_{1}+x_{2}(i)h_{33}(i)e_{2}%
=\lambda_{1}\left(  e_{1}+y_{2}^{\prime}(i)e_{2}\right)  $ and $h_{12}%
(i)e_{0}+x_{2}^{\prime}(i)h_{33}(i)e_{2}=\lambda_{2}\left(  e_{0}%
+y_{2}(i)e_{2}\right)  ,$ for some $\lambda_{1},\lambda_{2}\in E.$ Hence
$h\left(  e_{\tau_{i}}\mathrm{Fil}^{j}\mathcal{D}\left(  \underline
{a},\underline{x}\right)  \right)  =e_{\tau_{i}}\mathrm{Fil}^{j}%
\mathcal{D}\left(  \underline{a},\underline{y}\right)  $ if and only if
$x_{2}(i)=y_{2}^{\prime}(i)$ and $x_{2}^{\prime}(i)=y_{2}(i).$

$\left(  c\right)  $ If the filtration of $e_{\tau_{i}}\mathrm{Fil}%
^{j}\mathcal{D}\left(  \underline{a},\underline{x}\right)  $ is given by
formula (\ref{filtr'}). Then $h\left(  e_{\tau_{i}}\mathrm{Fil}^{j}%
\mathcal{D}\left(  \underline{a},\underline{x}\right)  \right)  =e_{\tau_{i}%
}\mathrm{Fil}^{j}\mathcal{D}\left(  \underline{a},\underline{y}\right)  $ is
equivalent to%
\[
E\left(  e_{0}+y_{1}(i)e_{1}+y_{2}^{\prime\prime}(i)e_{2}\right)  =E\left(
x_{1}(i)h_{12}(i)e_{0}+h_{21}(i)e_{1}+x_{2}^{\prime\prime}(i)h_{33}%
(i)e_{2}\right)  ,\
\]
which is in turn equivalent to $h_{21}(i)=y_{1}(i)x_{1}(i)h_{12}%
(i)\ $and\ $x_{2}^{\prime\prime}(i)h_{33}(i)=y_{2}^{\prime\prime}%
(i)x_{1}(i)h_{12}(i).$ Hence $h\left(  e_{\tau_{i}}\mathrm{Fil}^{j}%
\mathcal{D}\left(  \underline{a},\underline{x}\right)  \right)  =e_{\tau_{i}%
}\mathrm{Fil}^{j}\mathcal{D}\left(  \underline{a},\underline{y}\right)  $ if
and only if $x_{1}(i)=y_{1}(i)=1$ and $x_{2}^{\prime\prime}(i)=y_{2}%
^{\prime\prime}(i).$

$\left(  3\right)  $ If $\mathrm{Nm}_{\varphi}(\vec{a})=\mathrm{Nm}_{\varphi
}(\vec{c}_{1}),$ $\mathrm{Nm}_{\varphi}(\vec{b})=\mathrm{Nm}_{\varphi}(\vec
{a}_{1}),$ $\mathrm{Nm}_{\varphi}(\vec{c})=\mathrm{Nm}_{\varphi}(\vec{b}%
_{1}).$ In this case equations $\left(  \ref{2'}\right)  $ imply that
\[
H=\left(
\begin{array}
[c]{ccc}%
\vec{0} & \vec{h}_{12} & \vec{0}\\
\vec{0} & \vec{0} & \vec{h}_{23}\\
\vec{h}_{31} & \vec{0} & \vec{0}%
\end{array}
\right)
\]
and $\left(  he_{0},he_{1},he_{2}\right)  =\left(  \vec{h}_{31}e_{2},\vec
{h}_{12}e_{0},\vec{h}_{23}e_{1}\right)  \allowbreak.$

$\left(  a\right)  $ If the filtration of $e_{\tau_{i}}\mathrm{Fil}%
^{j}\mathcal{D}\left(  \underline{a},\underline{x}\right)  $ is given by
formula (\ref{1}). Then $h\left(  e_{\tau_{i}}\mathrm{Fil}^{j}\mathcal{D}%
\left(  \underline{a},\underline{x}\right)  \right)  =e_{\tau_{i}}%
\mathrm{Fil}^{j}\mathcal{D}\left(  \underline{a},\underline{y}\right)  $ is
equivalent to%
\begin{align}
E\left(  e_{0}+y_{1}(i)e_{1}+y_{2}^{\prime\prime}(i)e_{2}\right)   &
=E\left(  x_{1}(i)h_{12}(i)e_{0}+x_{2}^{\prime\prime}(i)h_{23}(i)e_{1}%
+h_{31}(i)e_{2}\right)  ,\label{12'}\\
\text{where }x_{2}^{\prime\prime}(i)  &  =x_{2}(i)+x_{1}(i)x_{2}^{\prime
}(i)\ \text{and }y_{2}^{\prime\prime}(i)=y_{2}(i)+y_{1}(i)y_{2}^{\prime
}(i),\text{\ and}\nonumber\\
E\left(  e_{0}+y_{2}(i)e_{2}\right)  \oplus E\left(  e_{1}+y_{2}^{\prime
}(i)e_{2}\right)   &  =E\left(  x_{2}(i)h_{23}(i)e_{1}+h_{31}(i)e_{2}\right)
\oplus E\left(  h_{12}(i)e_{0}+x_{2}^{\prime}(i)h_{23}(i)e_{1}\right)  .
\label{13'}%
\end{align}
Equation $\left(  \ref{12'}\right)  $ is equivalent to $x_{2}^{\prime\prime
}(i)h_{23}(i)e_{1}+h_{31}(i)e_{2}=y_{1}(i)x_{1}h_{12}(i)e_{1}+y_{2}%
^{\prime\prime}(i)x_{1}(i)h_{12}(i)e_{2}$ if and only if
\begin{align*}
x_{2}^{\prime\prime}(i)h_{23}(i)  &  =y_{1}(i)x_{1}(i)h_{12}(i),\\
h_{31}(i)  &  =y_{2}^{\prime\prime}(i)x_{1}(i)h_{12}(i).
\end{align*}
Equation $\left(  \ref{13'}\right)  $ is equivalent to%
\begin{align*}
x_{2}(i)h_{23}(i)e_{1}+h_{31}(i)e_{2}  &  =\mu_{1}e_{1}+\mu_{1}y_{2}^{\prime
}(i)e_{2},\\
h_{12}(i)e_{0}+x_{2}^{\prime}(i)h_{23}(i)e_{1}  &  =\lambda_{2}e_{0}+\mu
_{2}e_{1}+\left(  \mu_{2}y_{2}^{\prime}(i)+\lambda_{2}y_{2}(i)\right)  e_{2}%
\end{align*}
for some $\lambda_{i},\mu_{i}\in E.$We have$\ \lambda_{2}=h_{12}(i),\ \mu
_{1}=x_{2}(i)h(i)_{23},\ \mu_{2}=x_{2}^{\prime}(i)h_{23}(i),\ x_{2}%
(i)h_{23}(i)y_{2}^{\prime}(i)=h_{31}(i),$

$x_{2}^{\prime}(i)h_{23}(i)y_{2}^{\prime}(i)+h_{12}(i)y_{2}(i)=0,\ x_{2}%
^{\prime\prime}(i)h_{23}(i)=y_{1}(i)x_{1}(i)h_{12}(i),\ h_{31}(i)=y_{2}%
^{\prime\prime}(i)x_{1}(i)h_{12}(i)$ so $y_{2}^{\prime}(i)x_{2}(i)h_{23}%
(i)=h_{31}(i),$ $h_{23}(i)y_{2}^{\prime}(i)x_{2}^{\prime}(i)+h_{12}%
(i)y_{2}(i)=0,$ $x_{2}^{\prime\prime}(i)h_{23}(i)=y_{1}(i)x_{1}(i)h_{12}(i),$
and $h_{31}(i)=y_{2}^{\prime\prime}(i)x_{1}(i)h_{12}(i).$

$\left(  i\right)  $ If $x_{2}^{\prime}(i)=0,$ then%
\begin{align*}
x_{2}(i)h_{23}(i)  &  =h_{31}(i)\\
y_{2}(i)  &  =0\\
x_{2}(i)h_{23}(i)  &  =y_{1}(i)x_{1}(i)h_{12}(i)\\
h_{31}(i)  &  =y_{1}(i)x_{1}(i)h_{12}(i)\\
x_{1}(i)x_{2}(i)y_{1}(i)  &  \neq0
\end{align*}
so we must have $x_{2}(i)=1,$ $y_{1}(i)x_{1}(i)\neq0,$ $h_{23}(i)=y_{1}%
(i)x_{1}(i)h_{12}(i),$ and $h_{31}(i)=h_{23}(i).$\ Then $h\left(  e_{\tau_{i}%
}\mathrm{Fil}^{j}\mathcal{D}\left(  \underline{a},\underline{x}\right)
\right)  =\mathrm{Fil}^{j}e_{\tau_{i}}\mathcal{D}\left(  \underline
{a},\underline{y}\right)  $ is equivalent to%
\[
x_{2}^{\prime}(i)=y_{2}(i)=0,\ x_{2}(i)=y_{2}^{\prime}(i)=1,\ x_{1}%
(i)y_{1}(i)\neq0.
\]
$\left(  ii\right)  $ If $x_{2}^{\prime}(i)=1.$ Then
\begin{align*}
y_{2}^{\prime}(i)x_{2}(i)h_{23}(i)  &  =h_{31}(i)\\
h_{23}(i)y_{2}^{\prime}(i)x_{2}^{\prime}(i)+h_{12}(i)y_{2}(i)  &  =0\\
x_{2}^{\prime\prime}(i)h_{23}(i)  &  =y_{1}(i)x_{1}(i)h_{12}(i)\\
h_{31}(i)  &  =y_{2}^{\prime\prime}(i)x_{1}(i)h_{12}(i)
\end{align*}
which implies that $y_{2}^{\prime}(i)=x_{2}(i)=y_{2}(i)=1$ and
\begin{align*}
h_{31}(i)  &  =h_{23}(i)\\
h_{23}(i)  &  =h_{12}(i)\\
\left(  x_{1}(i)+1\right)  h_{23}(i)  &  =y_{1}(i)x_{1}(i)h_{12}(i)\\
h_{31}(i)  &  =\left(  y_{1}(i)+1\right)  x_{1}(i)h_{12}(i),
\end{align*}
so we need$\ 0=\left(  y_{1}(i)+1\right)  x_{1}(i)+y_{2}(i).$ Then $h\left(
e_{\tau_{i}}\mathrm{Fil}^{j}\mathcal{D}\left(  \underline{a},\underline
{x}\right)  \right)  =\mathrm{Fil}^{j}e_{\tau_{i}}\mathcal{D}\left(
\underline{a},\underline{y}\right)  $ is equivalent to%
\[
x_{2}^{\prime}(i)=y_{2}^{\prime}(i)=x_{2}(i)=y_{2}(i)\neq0,\ \text{and
}\left(  y_{1}(i)+y_{2}(i)\right)  x_{1}(i)+y_{2}(i)x_{2}(i)=0.
\]

Summary:

$x_{2}^{\prime}(i)=y_{2}(i)=x_{2}(i)=0,\ y_{2}^{\prime}(i)=1,\ x_{1}%
(i)y_{1}(i)\neq0,$ or

$x_{2}^{\prime}(i)=y_{2}^{\prime}(i)=x_{2}(i)=y_{2}(i)=1,\ $and $\left(
y_{1}(i)+1\right)  x_{1}(i)+1=0.$

$\left(  b\right)  $ If the filtration of $e_{\tau_{i}}\mathrm{Fil}%
^{j}\mathcal{D}\left(  \underline{a},\underline{x}\right)  $ is given by
formula (\ref{fil1}). Then $h\left(  e_{\tau_{i}}\mathrm{Fil}^{j}%
\mathcal{D}\left(  \underline{a},\underline{x}\right)  \right)  =e_{\tau_{i}%
}\mathrm{Fil}^{j}\mathcal{D}\left(  \underline{a},\underline{y}\right)  $ is
equivalent to%
\[
E\left(  e_{0}+y_{2}(i)e_{2}\right)  \oplus E\left(  e_{1}+y_{2}^{\prime
}(i)e_{2}\right)  =E\left(  x_{2}(i)h_{23}(i)e_{1}+h_{31}(i)e_{2}\right)
\oplus E\left(  h_{12}(i)e_{0}+x_{2}^{\prime}(i)h_{23}(i)e_{1}\right)
\]
which is in turn equivalent to
\begin{align*}
x_{2}(i)h_{23}(i)e_{1}+h_{31}(i)e_{2}  &  =\mu_{1}e_{1}+\mu_{1}y_{2}^{\prime
}(i)e_{2},\\
h_{12}(i)e_{0}+x_{2}^{\prime}(i)h_{23}(i)e_{1}  &  =\lambda_{2}e_{0}+\mu
_{2}e_{1}+\left(  \mu_{2}y_{2}^{\prime}(i)+\lambda_{2}y_{2}(i)\right)  e_{2}%
\end{align*}
for some $\lambda_{i},\mu_{i}\in E.$ We have $x_{2}(i)y_{2}^{\prime}%
(i)h_{23}(i)=h_{31}(i),$ and $x_{2}^{\prime}(i)h_{23}(i)y_{2}^{\prime
}(i)+h_{12}(i)y_{2}(i)=0.$ Then $h\left(  e_{\tau_{i}}\mathrm{Fil}%
^{j}\mathcal{D}\left(  \underline{a},\underline{x}\right)  \right)
=e_{\tau_{i}}\mathrm{Fil}^{j}\mathcal{D}\left(  \underline{a},\underline
{y}\right)  $ if and only if $x_{2}(i)=y_{2}^{\prime}(i)=1$ and $x_{2}%
^{\prime}(i)=y_{2}(i).$

$\left(  c\right)  $ If the filtration of $e_{\tau_{i}}\mathrm{Fil}%
^{j}\mathcal{D}\left(  \underline{a},\underline{x}\right)  $ is given by
formula (\ref{filtr'}). Then $h\left(  e_{\tau_{i}}\mathrm{Fil}^{j}%
\mathcal{D}\left(  \underline{a},\underline{x}\right)  \right)  =e_{\tau_{i}%
}\mathrm{Fil}^{j}\mathcal{D}\left(  \underline{a},\underline{y}\right)  $ is
equivalent to $E\left(  e_{0}+y_{1}(i)e_{1}+y_{2}^{\prime\prime}%
(i)e_{2}\right)  =E\left(  x_{1}(i)h_{12}(i)e_{0}+x_{2}^{\prime\prime
}(i)h_{23}(i)e_{1}+h_{31}(i)e_{2}\right)  ,$ and the latter is equivalent to
$x_{2}^{\prime\prime}(i)h_{23}(i)=y_{1}(i)x_{1}(i)h_{12}(i),$ and
$h_{31}(i)=y_{2}^{\prime\prime}(i)x_{1}(i)h_{12}(i).$ Then $h\left(
e_{\tau_{i}}\mathrm{Fil}^{j}\mathcal{D}\left(  \underline{a},\underline
{x}\right)  \right)  =e_{\tau_{i}}\mathrm{Fil}^{j}\mathcal{D}\left(
\underline{a},\underline{y}\right)  $ if and only if $y_{2}^{\prime\prime
}(i)=x_{1}(i)=1$ and $x_{2}^{\prime\prime}(i)=y_{1}(i).$

$\left(  4\right)  $ If $\mathrm{Nm}_{\varphi}(\vec{a})=\mathrm{Nm}_{\varphi
}(\vec{a}_{1}),$ $\mathrm{Nm}_{\varphi}(\vec{b})=\mathrm{Nm}_{\varphi}(\vec
{c}_{1}),$ $\mathrm{Nm}_{\varphi}(\vec{c})=\mathrm{Nm}_{\varphi}(\vec{b}_{1})$
then%
\[
H=\left(
\begin{array}
[c]{ccc}%
\vec{h}_{11} & \vec{0} & \vec{0}\\
\vec{0} & \vec{0} & \vec{h}_{23}\\
\vec{0} & \vec{h}_{32} & \vec{0}%
\end{array}
\right)
\]
and%
\[
\left(  he_{0},he_{1},he_{2}\right)  =\allowbreak\left(
\begin{array}
[c]{ccc}%
\vec{h}_{11}e_{0}, & \vec{h}_{32}e_{2}, & \vec{h}_{23}e_{1}%
\end{array}
\right)  .
\]

$\left(  a\right)  $ If the filtration of $e_{\tau_{i}}\mathrm{Fil}%
^{j}\mathcal{D}\left(  \underline{a},\underline{x}\right)  $ is given by
formula (\ref{1}). Then $h\left(  e_{\tau_{i}}\mathrm{Fil}^{j}\mathcal{D}%
\left(  \underline{a},\underline{x}\right)  \right)  =e_{\tau_{i}}%
\mathrm{Fil}^{j}\mathcal{D}\left(  \underline{a},\underline{y}\right)  $ is
equivalent to%
\begin{align}
E\left(  e_{0}+y_{1}(i)e_{1}+y_{2}^{\prime\prime}(i)e_{2}\right)   &
=E\left(  h_{11}(i)e_{0}+x_{2}^{\prime\prime}(i)h_{23}(i)e_{1}+x_{1}%
(i)h_{32}(i)e_{2}\right)  ,\label{14'}\\
\text{where }x_{2}^{\prime\prime}(i)  &  =x_{2}(i)+x_{1}(i)x_{2}^{\prime
}(i)\ \text{and }y_{2}^{\prime\prime}(i)=y_{2}(i)+y_{1}(i)y_{2}^{\prime
}(i),\text{\ and}\nonumber\\
E\left(  e_{0}+y_{2}(i)e_{2}\right)  \oplus E\left(  e_{1}+y_{2}^{\prime
}(i)e_{2}\right)   &  =E\left(  h_{11}(i)e_{0}+x_{2}(i)h_{23}(i)e_{1}\right)
\oplus E\left(  x_{2}^{\prime}(i)h_{23}(i)e_{1}+h_{32}(i)e_{2}\right)  .
\label{15'}%
\end{align}
Equation $\left(  \ref{14'}\right)  $ is equivalent to $x_{2}^{\prime\prime
}(i)h_{32}(i)e_{1}+x_{1}(i)h_{32}(i)e_{2}=y_{1}(i)h_{11}(i)e_{1}+y_{2}%
^{\prime\prime}(i)h_{11}(i)e_{2}$ or equivalently%
\begin{align*}
x_{2}^{\prime\prime}(i)h_{23}(i)  &  =y_{1}(i)h_{11}(i),\\
x_{1}(i)h_{32}(i)  &  =y_{2}^{\prime\prime}(i)h_{11}(i).
\end{align*}
Equation $\left(  \ref{15'}\right)  $ is equivalent to%
\begin{align*}
h_{11}(i)e_{0}+x_{2}(i)h_{23}(i)e_{1}  &  =\lambda_{1}e_{0}+\mu_{1}%
e_{1}+\left(  \mu_{1}y_{2}^{\prime}(i)+\lambda_{1}y_{2}(i)\right)  e_{2},\\
x_{2}^{\prime}(i)h_{23}(i)e_{1}+h_{32}(i)e_{2}  &  =\mu_{2}e_{1}+\mu_{2}%
y_{2}^{\prime}(i)e_{2}%
\end{align*}
for some $\lambda_{i},\mu_{i}\in E.$ Then $x_{2}^{\prime\prime}(i)h_{23}%
(i)=y_{1}(i)h_{11}(i),\ x_{1}(i)h_{32}(i)=y_{2}^{\prime\prime}(i)h_{11}%
(i),\ y_{2}^{\prime}(i)x_{2}(i)h_{23}(i)+y_{2}(i)h_{11}(i)=0,$ and
$h_{32}(i)=y_{2}^{\prime}(i)x_{2}^{\prime}(i)h_{23}(i).$ Hence $y_{2}^{\prime
}(i)=x_{2}^{\prime}(i)=1,$ $h_{23}(i)=h_{32}(i)$ and
\begin{align*}
\left(  x_{1}(i)+x_{2}(i)\right)  h_{32}(i)  &  =y_{1}(i)h_{11}(i),\\
x_{1}(i)h_{32}(i)  &  =\left(  y_{1}(i)+y_{2}(i)\right)  h_{11}(i),\\
x_{2}(i)h_{32}(i)  &  =-y_{2}(i)h_{11}(i).
\end{align*}
if and only if $y_{2}^{\prime}(i)=x_{2}^{\prime}(i)=1,$ $h_{23}(i)=h_{32}(i),$
$x_{1}(i)h_{32}(i)=\left(  y_{1}(i)+y_{2}(i)\right)  h_{11}(i),\ $and
$x_{2}(i)h_{32}(i)=-y_{2}(i)h_{11}(i).$

$\left(  i\right)  \ $If $x_{2}(i)=1$ then $y_{2}(i)=1$ and $x_{1}%
(i)+y_{1}(i)+1=0,\ h_{32}(i)=-h_{11}(i).$ Then $h\left(  \mathrm{Fil}%
^{j}e_{\tau_{i}}\mathcal{D}\left(  \underline{a},\underline{x}\right)
\right)  =\mathrm{Fil}^{j}e_{\tau_{i}}\mathcal{D}\left(  \underline
{a},\underline{y}\right)  $ is equivalent to%
\[
y_{2}^{\prime}(i)=x_{2}^{\prime}(i)=1,\ x_{2}(i)y_{2}(i)\neq0,\ \text{and}%
\ x_{1}(i)y_{2}(i)+x_{2}(i)\left(  y_{1}(i)+y_{2}(i)\right)  =0.
\]
$\left(  ii\right)  $ If $x_{2}(i)=0,$ then $y_{2}^{\prime}(i)=x_{2}^{\prime
}(i)=1,\ y_{2}(i)=0,$ and
\[
x_{1}(i)h_{32}(i)=y_{1}(i)h_{11}(i),
\]
so we need $x_{1}(i)\neq0$ if and only if $y_{1}(i)\neq0$ and $h\left(
\mathrm{Fil}^{j}e_{\tau_{i}}\mathcal{D}\left(  \underline{a},\underline
{x}\right)  \right)  =\mathrm{Fil}^{j}e_{\tau_{i}}\mathcal{D}\left(
\underline{a},\underline{y}\right)  $ is equivalent to%
\[
x_{2}(i)=y_{2}(i)=0,\ x_{2}^{\prime}(i)=y_{2}^{\prime}(i)=1,\ \text{and\ }%
x_{1}(i)\neq0\ \text{if\ and\ only\ if\ }y_{1}(i)\neq0.
\]

$\left(  b\right)  $ If the filtration of $e_{\tau_{i}}\mathrm{Fil}%
^{j}\mathcal{D}\left(  \underline{a},\underline{x}\right)  $ is given by
formula (\ref{fil1}). Then $h\left(  e_{\tau_{i}}\mathrm{Fil}^{j}%
\mathcal{D}\left(  \underline{a},\underline{x}\right)  \right)  =e_{\tau_{i}%
}\mathrm{Fil}^{j}\mathcal{D}\left(  \underline{a},\underline{y}\right)  $ is
equivalent to%
\[
E\left(  e_{0}+y_{2}(i)e_{2}\right)  \oplus E\left(  e_{1}+y_{2}^{\prime
}(i)e_{2}\right)  =E\left(  h_{11}(i)e_{0}+x_{2}(i)h_{23}(i)e_{1}\right)
\oplus E\left(  x_{2}^{\prime}(i)h_{23}(i)e_{1}+h_{32}(i)e_{2}\right)
\]
and this is equivalent to
\begin{align*}
h_{11}(i)e_{0}+x_{2}(i)h_{23}(i)e_{1}  &  =\lambda_{1}e_{0}+\mu_{1}%
e_{1}+\left(  \mu_{1}y_{2}^{\prime}(i)+\lambda_{1}y_{2}(i)\right)  e_{2},\\
x_{2}^{\prime}(i)h_{23}(i)e_{1}+h_{32}(i)e_{2}  &  =\mu_{2}e_{1}+\mu_{2}%
y_{2}^{\prime}(i)e_{2}%
\end{align*}
for some $\lambda_{i},\mu_{i}\in E.$ We have $x_{2}(i)y_{2}^{\prime}%
(i)h_{23}(i)+h_{11}(i)y_{2}(i)=0,$ and $h_{32}(i)=x_{2}^{\prime}%
(i)y_{2}^{\prime}(i)h_{23}(i).$ Hence $h\left(  \mathrm{Fil}^{j}e_{\tau_{i}%
}\mathcal{D}\left(  \underline{a},\underline{x}\right)  \right)
=\mathrm{Fil}^{j}e_{\tau_{i}}\mathcal{D}\left(  \underline{a},\underline
{y}\right)  $ if and only if $x_{2}^{\prime}(i)=y_{2}^{\prime}(i)=1$ and
$x_{2}(i)=y_{2}(i).$

$\left(  c\right)  $ If the filtration of $e_{\tau_{i}}\mathrm{Fil}%
^{j}\mathcal{D}\left(  \underline{a},\underline{x}\right)  $ is given by
formula (\ref{filtr'}). Then $h\left(  e_{\tau_{i}}\mathrm{Fil}^{j}%
\mathcal{D}\left(  \underline{a},\underline{x}\right)  \right)  =e_{\tau_{i}%
}\mathrm{Fil}^{j}\mathcal{D}\left(  \underline{a},\underline{y}\right)  $ is
equivalent to%
\[
E\left(  e_{0}+y_{1}(i)e_{1}+y_{2}^{\prime\prime}(i)e_{2}\right)  =E\left(
h_{11}(i)e_{0}+x_{2}^{\prime\prime}(i)h_{23}(i)e_{1}+x_{1}(i)h_{32}%
(i)e_{2}\right)
\]
and this is equivalent to$\ x_{2}^{\prime\prime}(i)h_{23}(i)=y_{1}%
(i)h_{11}(i),$ and$\ x_{1}(i)h_{32}(i)=y_{2}^{\prime\prime}(i)h_{11}(i).$
Hence $h\left(  \mathrm{Fil}^{j}e_{\tau_{i}}\mathcal{D}\left(  \underline
{a},\underline{x}\right)  \right)  =\mathrm{Fil}^{j}e_{\tau_{i}}%
\mathcal{D}\left(  \underline{a},\underline{y}\right)  $ if and only if
$x_{2}^{\prime\prime}(i)=y_{1}(i)$ and $x_{1}(i)=y_{2}^{\prime\prime}(i).$

$\left(  5\right)  $ If $\mathrm{Nm}_{\varphi}(\vec{a})=\mathrm{Nm}_{\varphi
}(\vec{c}_{1}),$ $\mathrm{Nm}_{\varphi}(\vec{b})=\mathrm{Nm}_{\varphi}(\vec
{b}_{1}),$ $\mathrm{Nm}_{\varphi}(\vec{c})=\mathrm{Nm}_{\varphi}(\vec{a}%
_{1}).$ Then
\[
H=\left(
\begin{array}
[c]{ccc}%
\vec{0} & \vec{0} & \vec{h}_{13}\\
\vec{0} & \vec{h}_{22} & \vec{0}\\
\vec{h}_{31} & \vec{0} & \vec{0}%
\end{array}
\right)
\]
and%
\[
\left(  he_{0},he_{1},he_{2}\right)  =\allowbreak\left(
\begin{array}
[c]{ccc}%
\vec{h}_{31}e_{2}, & \vec{h}_{22}e_{1}, & \vec{h}_{13}e_{0}%
\end{array}
\right)
\]
$\left(  a\right)  $ If the filtration of $e_{\tau_{i}}\mathrm{Fil}%
^{j}\mathcal{D}\left(  \underline{a},\underline{x}\right)  $ is given by
formula (\ref{1}). Then $h\left(  e_{\tau_{i}}\mathrm{Fil}^{j}\mathcal{D}%
\left(  \underline{a},\underline{x}\right)  \right)  =e_{\tau_{i}}%
\mathrm{Fil}^{j}\mathcal{D}\left(  \underline{a},\underline{y}\right)  $ is
equivalent to%
\begin{align}
E\left(  e_{0}+y_{1}(i)e_{1}+y_{2}^{\prime\prime}(i)e_{2}\right)   &
=E\left(  x_{2}^{\prime\prime}(i)h_{13}(i)e_{0}+x_{1}(i)h_{22}(i)e_{1}%
+h_{31}(i)e_{2}\right)  ,\label{16'}\\
\text{where }x_{2}^{\prime\prime}(i)  &  =x_{2}(i)+x_{1}(i)x_{2}^{\prime
}(i)\ \text{and }y_{2}^{\prime\prime}(i)=y_{2}(i)+y_{1}(i)y_{2}^{\prime
}(i),\text{\ and}\nonumber\\
E\left(  e_{0}+y_{2}(i)e_{2}\right)  \oplus E\left(  e_{1}+y_{2}^{\prime
}(i)e_{2}\right)   &  =E\left(  x_{2}(i)h_{13}(i)e_{0}+h_{31}(i)e_{2}\right)
\oplus E\left(  x_{2}^{\prime}(i)h_{13}(i)e_{0}+h_{22}(i)e_{1}\right)  .
\label{17'}%
\end{align}
Equation $\left(  \ref{16'}\right)  $ is equivalent to $x_{1}(i)h_{22}%
(i)e_{1}+h_{31}(i)e_{2}=y_{1}(i)x_{2}^{\prime\prime}(i)h_{13}(i)e_{1}%
+y_{2}^{\prime\prime}(i)x_{2}^{\prime\prime}(i)h_{13}(i)e_{2}$ if and only if
\begin{align*}
x_{1}(i)h_{22}(i)  &  =y_{1}(i)x_{2}^{\prime\prime}(i)h_{13}(i),\\
h_{31}(i)  &  =y_{2}^{\prime\prime}(i)x_{2}^{\prime\prime}(i)h_{13}(i).
\end{align*}
Equation $\left(  \ref{17'}\right)  $ is equivalent to%
\begin{align*}
x_{2}(i)h_{13}(i)e_{0}+h_{31}(i)e_{2}  &  =\lambda_{1}e_{0}+\lambda_{1}%
y_{2}(i)e_{2},\\
x_{2}^{\prime}(i)h_{13}(i)e_{0}+h_{22}(i)e_{1}  &  =\lambda_{2}e_{0}+\mu
_{2}e_{1}+\left(  \mu_{2}y_{2}^{\prime}(i)+\lambda_{2}y_{2}(i)\right)  e_{2}%
\end{align*}
for some $\lambda_{i},\mu_{i}\in E.$ So,$\ h_{31}(i)=x_{2}(i)y_{2}%
(i)h_{13}(i),\ y_{2}^{\prime}(i)h_{22}(i)=-x_{2}^{\prime}(i)y_{2}%
(i)h_{13}(i),\ x_{1}(i)h_{22}(i)=y_{1}(i)x_{2}^{\prime\prime}(i)h_{13}%
(i),\ h_{31}(i)=y_{2}^{\prime\prime}(i)x_{2}^{\prime\prime}(i)h_{13}(i)\ $

and we must have $x_{2}(i)=y_{2}(i)=y_{2}^{\prime\prime}(i)=x_{2}%
^{\prime\prime}(i)=1,\ y_{2}^{\prime}(i)h_{22}(i)=-x_{2}^{\prime}%
(i)y_{2}(i)h_{13}(i),\ h_{31}(i)=h_{13}(i),\ $and$\ x_{1}(i)h_{22}%
(i)=y_{1}(i)h_{13}(i)$ $\left(  i\right)  $ $x_{2}^{\prime}(i)=1$ then
$y_{2}^{\prime}(i)=1$ and$\ h_{22}(i)=-h_{13}(i),\ h_{31}(i)=h_{13}%
(i),\ x_{1}(i)h_{22}(i)=y_{1}(i)\left(  x_{1}(i)+1\right)  h_{13}%
(i),\ $and$\ 1=\left(  y_{1}(i)+1\right)  \left(  x_{1}(i)+1\right)  .\left(
ia\right)  $ $x_{1}(i)\neq0.$ Then $x_{2}^{\prime}(i)=y_{2}^{\prime}(i)=1$ and
we need $y_{1}(i)\left(  x_{1}(i)+1\right)  h_{13}(i)=-x_{1}(i)h_{13}(i)$
which holds because $1=\left(  y_{1}(i)+1\right)  \left(  x_{1}(i)+1\right)
.$ Then $h\left(  \mathrm{Fil}^{j}e_{\tau_{i}}\mathcal{D}\left(  \underline
{a},\underline{x}\right)  \right)  =\mathrm{Fil}^{j}e_{\tau_{i}}%
\mathcal{D}\left(  \underline{a},\underline{y}\right)  $ is equivalent to
\[
x_{2}^{\prime}(i)=y_{2}^{\prime}(i)=x_{2}(i)=y_{2}(i)=1,\ x_{1}(i)y_{1}%
(i)\neq0,\ \text{and }1=\left(  y_{1}(i)+1\right)  \left(  x_{1}(i)+1\right)
.
\]
$\left(  ib\right)  $ If $x_{1}(i)=0.$ Then $x_{2}^{\prime}(i)=y_{2}^{\prime
}(i)=y_{2}(i)=x_{2}(i)=1$ and$\ h_{22}(i)=-h_{13}(i),$ $h_{31}(i)=h_{13}(i).$
Then $h\left(  \mathrm{Fil}^{j}e_{\tau_{i}}\mathcal{D}\left(  \underline
{a},\underline{x}\right)  \right)  =\mathrm{Fil}^{j}e_{\tau_{i}}%
\mathcal{D}\left(  \underline{a},\underline{y}\right)  $ is equivalent to%
\[
x_{1}(i)=y_{1}(i)=0,\ x_{2}^{\prime}(i)=y_{2}^{\prime}(i)=x_{2}(i)=y_{2}%
(i)=1.
\]

$\left(  ii\right)  $ If $x_{2}^{\prime}(i)=0=y_{2}^{\prime}(i),$
then$\ x_{2}(i)=y_{2}(i)=1$ and$\ x_{1}(i)h_{22}(i)=y_{1}(i)h_{13}%
(i),\ h_{31}(i)=h_{13}(i).$

$\left(  iia\right)  $ If $x_{1}(i)=0.$ Then $x_{2}^{\prime}(i)=y_{2}^{\prime
}(i)=0,$ $x_{2}(i)=y_{2}(i)=1,\ y_{1}(i)h_{13}(i)=0$ and $h_{31}%
(i)=h_{13}(i).$ Then $h\left(  \mathrm{Fil}^{j}e_{\tau_{i}}\mathcal{D}\left(
\underline{a},\underline{x}\right)  \right)  =\mathrm{Fil}^{j}e_{\tau_{i}%
}\mathcal{D}\left(  \underline{a},\underline{y}\right)  $ is equivalent to%
\[
x_{1}(i)=y_{1}(i)=x_{2}^{\prime}(i)=y_{2}^{\prime}(i)=0,\ \text{and }%
x_{2}(i)=y_{2}(i)=1.
\]
$\left(  iib\right)  $\ If$\ x_{1}(i)\neq0,\ $then we must have $x_{2}%
^{\prime}(i)=y_{2}^{\prime}(i)=0,\ x_{2}(i)=y_{2}(i)=1,\ x_{1}(i)h_{22}%
(i)=y_{1}(i)h_{13}(i)\ $and $y_{1}(i)\neq0.$ Then $h\left(  \mathrm{Fil}%
^{j}e_{\tau_{i}}\mathcal{D}\left(  \underline{a},\underline{x}\right)
\right)  =\mathrm{Fil}^{j}e_{\tau_{i}}\mathcal{D}\left(  \underline
{a},\underline{y}\right)  $ is equivalent to
\[
x_{2}^{\prime}(i)=y_{2}^{\prime}(i)=0,\ x_{2}(i)=y_{2}(i)=1,\ \text{and }%
x_{1}(i)y_{1}(i)\neq0.
\]

$\left(  b\right)  $ If the filtration of $e_{\tau_{i}}\mathrm{Fil}%
^{j}\mathcal{D}\left(  \underline{a},\underline{x}\right)  $ is given by
formula (\ref{fil1}). Then $h\left(  e_{\tau_{i}}\mathrm{Fil}^{j}%
\mathcal{D}\left(  \underline{a},\underline{x}\right)  \right)  =e_{\tau_{i}%
}\mathrm{Fil}^{j}\mathcal{D}\left(  \underline{a},\underline{y}\right)  $ is
equivalent to%
\[
E\left(  e_{0}+y_{2}(i)e_{2}\right)  \oplus E\left(  e_{1}+y_{2}^{\prime
}(i)e_{2}\right)  =E\left(  x_{2}(i)h_{13}(i)e_{0}+h_{31}(i)e_{2}\right)
\oplus E\left(  x_{2}^{\prime}(i)h_{13}(i)e_{0}+h_{22}(i)e_{1}\right)
\]
which is equivalent to
\begin{align*}
x_{2}(i)h_{13}(i)e_{0}+h_{31}(i)e_{2}  &  =\lambda_{1}e_{0}+\lambda_{1}%
y_{2}(i)e_{2},\\
x_{2}^{\prime}(i)h_{13}(i)e_{0}+h_{22}(i)e_{1}  &  =\lambda_{2}e_{0}+\mu
_{2}e_{1}+\left(  \mu_{2}y_{2}^{\prime}(i)+\lambda_{2}y_{2}(i)\right)  e_{2}%
\end{align*}
for some $\lambda_{i},\mu_{i}\in E.$ Then $y_{2}(i)x_{2}(i)h_{13}%
(i)=h_{31}(i)$ and $h_{22}(i)y_{2}^{\prime}(i)+x_{2}^{\prime}(i)h_{13}%
(i)y_{2}(i)=0.$ Hence $h\left(  \mathrm{Fil}^{j}e_{\tau_{i}}\mathcal{D}\left(
\underline{a},\underline{x}\right)  \right)  =\mathrm{Fil}^{j}e_{\tau_{i}%
}\mathcal{D}\left(  \underline{a},\underline{y}\right)  $ if and only if
$y_{2}(i)=x_{2}(i)=1$ and $y_{2}^{\prime}(i)=x_{2}^{\prime}(i).$

$\left(  c\right)  $ If the filtration of $e_{\tau_{i}}\mathrm{Fil}%
^{j}\mathcal{D}\left(  \underline{a},\underline{x}\right)  $ is given by
formula (\ref{filtr'}). Then $h\left(  e_{\tau_{i}}\mathrm{Fil}^{j}%
\mathcal{D}\left(  \underline{a},\underline{x}\right)  \right)  =e_{\tau_{i}%
}\mathrm{Fil}^{j}\mathcal{D}\left(  \underline{a},\underline{y}\right)  $ is
equivalent to%
\[
E\left(  e_{0}+y_{1}(i)e_{1}+y_{2}^{\prime\prime}(i)e_{2}\right)  =E\left(
x_{2}^{\prime\prime}(i)h_{13}(i)e_{0}+x_{1}(i)h_{22}(i)e_{1}+h_{31}%
(i)e_{2}\right)
\]
which is equivalent to$\ x_{1}(i)h_{22}(i)=y_{1}(i)x_{2}^{\prime\prime
}(i)h_{13}(i),\ $and$\ h_{31}(i)=y_{2}^{\prime\prime}(i)x_{2}^{\prime\prime
}(i)h_{13}(i).\ $Hence%
\[
h\left(  \mathrm{Fil}^{j}e_{\tau_{i}}\mathcal{D}\left(  \underline
{a},\underline{x}\right)  \right)  =\mathrm{Fil}^{j}e_{\tau_{i}}%
\mathcal{D}\left(  \underline{a},\underline{y}\right)
\]
if and only if $\ y_{2}^{\prime\prime}(i)=x_{2}^{\prime\prime}(i)=1$ and
$y_{1}(i)=x_{1}(i).$

$\left(  6\right)  \ $If $\mathrm{Nm}_{\varphi}(\vec{a})=\mathrm{Nm}_{\varphi
}(\vec{b}_{1}),$ $\mathrm{Nm}_{\varphi}(\vec{b})=\mathrm{Nm}_{\varphi}(\vec
{c}_{1}),$ $\mathrm{Nm}_{\varphi}(\vec{c})=\mathrm{Nm}_{\varphi}(\vec{a}%
_{1}).$ Then
\[
H=\left(
\begin{array}
[c]{ccc}%
\vec{0} & \vec{0} & \vec{h}_{13}\\
\vec{h}_{21} & \vec{0} & \vec{0}\\
\vec{0} & \vec{h}_{32} & \vec{0}%
\end{array}
\right)
\]
and%
\[
\left(  he_{0},he_{1},he_{2}\right)  =\left(
\begin{array}
[c]{ccc}%
\vec{h}_{21}e_{1}, & \vec{h}_{32}e_{2}, & \vec{h}_{13}e_{0}%
\end{array}
\right)  .
\]
$\left(  a\right)  $ If the filtration of $e_{\tau_{i}}\mathrm{Fil}%
^{j}\mathcal{D}\left(  \underline{a},\underline{x}\right)  $ is given by
formula (\ref{1}). Then $h\left(  e_{\tau_{i}}\mathrm{Fil}^{j}\mathcal{D}%
\left(  \underline{a},\underline{x}\right)  \right)  =e_{\tau_{i}}%
\mathrm{Fil}^{j}\mathcal{D}\left(  \underline{a},\underline{y}\right)  $ is
equivalent to%
\begin{align}
E\left(  e_{0}+y_{1}(i)e_{1}+y_{2}^{\prime\prime}(i)e_{2}\right)   &
=E\left(  x_{2}^{\prime\prime}(i)h_{13}(i)e_{0}+h_{21}(i)e_{1}+x_{1}%
(i)h_{32}(i)e_{2}\right)  ,\label{18'}\\
\text{where }x_{2}^{\prime\prime}(i)  &  =x_{2}(i)+x_{1}(i)x_{2}^{\prime
}(i)\ \text{and }y_{2}^{\prime\prime}(i)=y_{2}(i)+y_{1}(i)y_{2}^{\prime
}(i),\text{\ and}\nonumber\\
E\left(  e_{0}+y_{2}(i)e_{2}\right)  \oplus E\left(  e_{1}+y_{2}^{\prime
}(i)e_{2}\right)   &  =E\left(  x_{2}(i)h_{13}(i)e_{0}+h_{21}(i)e_{1}\right)
\oplus E\left(  x_{2}^{\prime}(i)h_{13}(i)e_{0}+h_{32}(i)e_{2}\right)  .
\label{19'}%
\end{align}
Equation $\left(  \ref{18}\right)  $ is equivalent to $h_{21}(i)e_{1}%
+x_{1}(i)h_{32}(i)e_{2}=y_{1}(i)x_{2}^{\prime\prime}(i)h_{13}(i)e_{1}%
+y_{2}^{\prime\prime}(i)x_{2}^{\prime\prime}(i)h_{13}(i)e_{2}$ if and only if
\begin{align*}
h_{21}(i)  &  =y_{1}(i)x_{2}^{\prime\prime}(i)h_{13}(i)\\
x_{1}(i)h_{32}(i)  &  =y_{2}^{\prime\prime}(i)x_{2}^{\prime\prime}(i)h_{13}(i)
\end{align*}
and equation $\left(  \ref{19'}\right)  $ is equivalent to
\begin{align*}
x_{2}(i)h_{13}(i)e_{0}+h_{21}(i)e_{1}  &  =\lambda_{1}e_{0}+\mu_{1}%
e_{1}+\left(  \lambda_{1}y_{2}(i)+\mu_{1}y_{2}^{\prime}(i)\right)  e_{2}\\
x_{2}^{\prime}(i)h_{13}(i)e_{0}+h_{32}(i)e_{2}  &  =\lambda_{2}e_{0}%
+\lambda_{2}y_{2}(i)e_{2}%
\end{align*}
for some $\lambda_{i},\mu_{i}\in E.$ We have$\ h_{21}(i)=y_{1}(i)x_{2}%
^{\prime\prime}(i)h_{13}(i),\ x_{1}(i)h_{32}(i)=y_{2}^{\prime\prime}%
(i)x_{2}^{\prime\prime}(i)h_{13}(i),\ x_{2}(i)y_{2}(i)h_{13}(i)+y_{2}^{\prime
}(i)h_{21}(i)=0,\ $and$\ x_{2}^{\prime}(i)y_{2}(i)h_{13}(i)=h_{32}(i).\ $Then
$x_{2}^{\prime}(i)=x_{2}^{\prime\prime}(i)=y_{2}(i)=1$ and$\ h_{21}%
(i)=y_{1}(i)h_{13}(i),\ h_{32}(i)=h_{13}(i),\ x_{1}(i)h_{32}(i)=y_{2}%
^{\prime\prime}(i)\left(  x_{1}(i)+x_{2}(i)\right)  h_{13}(i),\ x_{2}%
(i)y_{2}(i)h_{13}(i)+y_{2}^{\prime}(i)h_{21}(i)=0.\ $

$\left(  i\right)  $ If $y_{2}^{\prime}(i)=x_{2}(i)=0$, then we must have
$x_{2}^{\prime}(i)=y_{1}(i)=y_{2}(i)=1,$ $x_{1}(i)y_{1}(i)\neq0,$
$h_{21}(i)=y_{1}(i)x_{2}^{\prime\prime}(i)h_{13}(i),$ and $h_{32}%
(i)=y_{2}(i)h_{13}(i).$ Then $h\left(  e_{\tau_{i}}\mathrm{Fil}^{j}%
\mathcal{D}\left(  \underline{a},\underline{x}\right)  \right)  =e_{\tau_{i}%
}\mathrm{Fil}^{j}\mathcal{D}\left(  \underline{a},\underline{y}\right)  $ if
and only if
\[
x_{2}^{\prime}(i)=y_{2}(i)=1,\ y_{2}^{\prime}(i)=x_{2}(i)=0,\ x_{1}%
(i)y_{1}(i)\neq0.
\]
$\left(  ii\right)  $ If $y_{2}^{\prime}(i)=x_{2}^{\prime}(i)=1.$ Thenwe must
have $x_{2}(i)=y_{2}(i)=1,\ y_{1}(i)\left(  x_{1}(i)+1\right)  +1=0,\ h_{21}%
(i)=-h_{13}(i),$ and $h_{32}(i)=h_{13}(i).$ Hence $h\left(  e_{\tau_{i}%
}\mathrm{Fil}^{j}\mathcal{D}\left(  \underline{a},\underline{x}\right)
\right)  =e_{\tau_{i}}\mathrm{Fil}^{j}\mathcal{D}\left(  \underline
{a},\underline{y}\right)  $ if and only if%
\[
x_{2}^{\prime}(i)=y_{2}^{\prime}(i)=x_{2}(i)=y_{2}(i)=1,\ \text{and }%
y_{1}(i)\left(  x_{1}(i)+1\right)  +1=0.
\]
$\left(  b\right)  $ If the filtration of $e_{\tau_{i}}\mathrm{Fil}%
^{j}\mathcal{D}\left(  \underline{a},\underline{x}\right)  $ is given by
formula (\ref{fil1}). Then $h\left(  e_{\tau_{i}}\mathrm{Fil}^{j}%
\mathcal{D}\left(  \underline{a},\underline{x}\right)  \right)  =e_{\tau_{i}%
}\mathrm{Fil}^{j}\mathcal{D}\left(  \underline{a},\underline{y}\right)  $ is
equivalent to%
\[
E\left(  e_{0}+y_{2}(i)e_{2}\right)  \oplus E\left(  e_{1}+y_{2}^{\prime
}(i)e_{2}\right)  =E\left(  x_{2}(i)h_{13}(i)e_{0}+h_{21}(i)e_{1}\right)
\oplus E\left(  x_{2}^{\prime}(i)h_{13}(i)e_{0}+h_{32}(i)e_{2}\right)
\]
which is equivalent to
\begin{align*}
x_{2}(i)h_{13}(i)e_{0}+h_{21}(i)e_{1}  &  =\lambda_{1}e_{0}+\mu_{1}%
e_{1}+\left(  \lambda_{1}y_{2}(i)+\mu_{1}y_{2}^{\prime}(i)\right)  e_{2}\\
x_{2}^{\prime}(i)h_{13}(i)e_{0}+h_{32}(i)e_{2}  &  =\lambda_{2}e_{0}%
+\lambda_{2}y_{2}(i)e_{2}%
\end{align*}
for some $\lambda_{i},\mu_{i}\in E.$ We must have $y_{2}(i)x_{2}%
(i)h_{13}(i)+h_{21}(i)y_{2}^{\prime}(i)=0$ and $y_{2}(i)x_{2}^{\prime
}(i)h_{13}(i)=h_{32}(i).$ Then $h\left(  e_{\tau_{i}}\mathrm{Fil}%
^{j}\mathcal{D}\left(  \underline{a},\underline{x}\right)  \right)
=e_{\tau_{i}}\mathrm{Fil}^{j}\mathcal{D}\left(  \underline{a},\underline
{y}\right)  $ if and only if $y_{2}(i)=x_{2}^{\prime}(i)=1$ and $x_{2}%
(i)=y_{2}^{\prime}(i).$

$\left(  c\right)  $ If the filtration of $e_{\tau_{i}}\mathrm{Fil}%
^{j}\mathcal{D}\left(  \underline{a},\underline{x}\right)  $ is given by
formula (\ref{filtr'}). Then $h\left(  e_{\tau_{i}}\mathrm{Fil}^{j}%
\mathcal{D}\left(  \underline{a},\underline{x}\right)  \right)  =e_{\tau_{i}%
}\mathrm{Fil}^{j}\mathcal{D}\left(  \underline{a},\underline{y}\right)  $ is
equivalent to%
\[
E\left(  e_{0}+y_{1}(i)e_{1}+y_{2}^{\prime\prime}(i)e_{2}\right)  =E\left(
x_{2}^{\prime\prime}(i)h_{13}(i)e_{0}+h_{21}(i)e_{1}+x_{1}(i)h_{32}%
(i)e_{2}\right)
\]
which is equivalent to$\ h_{21}(i)=y_{1}(i)x_{2}^{\prime\prime}(i)h_{13}(i),$
and$\ x_{1}(i)h_{32}(i)=y_{2}^{\prime\prime}(i)x_{2}^{\prime\prime}%
(i)h_{13}(i).\ $Then $h\left(  e_{\tau_{i}}\mathrm{Fil}^{j}\mathcal{D}\left(
\underline{a},\underline{x}\right)  \right)  =e_{\tau_{i}}\mathrm{Fil}%
^{j}\mathcal{D}\left(  \underline{a},\underline{y}\right)  $ if and only if
$y_{1}(i)=x_{2}^{\prime\prime}(i)=1$ and $x_{1}(i)=y_{2}^{\prime\prime}%
(i).$\bigskip
\end{proof}

\begin{proof}
[Proof of Proposition \ref{semistable}]Let $\left(  D,\varphi,N\right)  $ be a
filtered $\left(  \varphi,N\right)  $-module and let $\left(  D,\varphi
\right)  $ be as in Proposition \ref{phi, fil}. Let $\underline{e}\ $be an
ordered basis such that $\mathrm{Mat}_{\underline{e}}\left(  \varphi\right)
=\mathrm{diag}\left(  \vec{a},\vec{b},\vec{c}\right)  .$ Let
$A:=[N]_{\underline{e}}=\left(  \vec{a}_{ij}\right)  .$ Since $N\varphi
=p\varphi N,$ we have $AP=pP\varphi\left(  A\right)  ,$ where $P:=\mathrm{Mat}%
_{\underline{e}}\left(  \varphi\right)  .$ Hence$\bigskip\allowbreak$%
\begin{align}
\vec{a}\cdot\vec{a}_{11}  &  =p\cdot\vec{a}\cdot\varphi\left(  \vec{a}%
_{11}\right)  ,\ \vec{b}\cdot\vec{a}_{12}=p\cdot\vec{a}\cdot\varphi\left(
\vec{a}_{12}\right)  ,\ \vec{c}\cdot\vec{a}_{13}=p\cdot\vec{a}\cdot
\varphi\left(  \vec{a}_{13}\right)  ,\label{14}\\
\vec{a}\cdot\vec{a}_{21}  &  =p\cdot\vec{b}\cdot\varphi\left(  \vec{a}%
_{21}\right)  ,\ \vec{b}\cdot\vec{a}_{22}=p\cdot\vec{b}\cdot\varphi\left(
\vec{a}_{22}\right)  ,\ \vec{c}\cdot\vec{a}_{23}=p\cdot\vec{b}\cdot
\varphi\left(  \vec{a}_{23}\right)  ,\label{18}\\
\vec{a}\cdot\vec{a}_{31}  &  =p\cdot\vec{c}\cdot\varphi\left(  \vec{a}%
_{31}\right)  ,\ \vec{b}\cdot\vec{a}_{32}=p\cdot\vec{c}\cdot\varphi\left(
\vec{a}_{32}\right)  ,\ \vec{c}\cdot\vec{a}_{33}=p\cdot\vec{c}\cdot
\varphi\left(  \vec{a}_{33}\right)  \label{17}%
\end{align}
for all $i=1,2,3.$ Since the coordinates of the vectors $\vec{a},\vec{b}%
,\vec{c}$ are nonzero, equations $\left(  \ref{14}\right)  ,\left(
\ref{18}\right)  ,\left(  \ref{17}\right)  $ imply that either all the
coordinates of the vectors $\vec{a}_{ij}$ are nonzero or $\vec{a}_{ij}=\vec
{0},$ for any $i,j\in\{1,2,3\}.$ We need the following lemma whose proof is straightforward.
\end{proof}

\begin{lemma}
\label{norm lemma}Let $\vec{\alpha},\vec{\beta}\in\left(  E^{\times}\right)
^{\mid\mathcal{\tau}\mid}.$ The equation $\vec{\alpha}\cdot\vec{\gamma}%
=\vec{\beta}\cdot\varphi(\vec{\gamma})$ has nonzero solutions $\vec{\gamma}\in
E^{\mid\mathcal{\tau}\mid}\ $if and only if $\mathrm{Nm}_{\varphi}(\vec
{\alpha})=\mathrm{Nm}_{\varphi}(\vec{\beta}).$ In this case, all the solutions
are \noindent$\vec{\gamma}=\gamma\left(  1,\frac{\alpha_{0}}{\beta_{0}}%
,\frac{\alpha_{0}\alpha_{1}}{\beta_{0}\beta_{1}},...,\frac{\alpha_{0}%
\alpha_{1}\cdots\alpha_{f-2}}{\beta_{0}\beta_{1}\cdots\beta_{f-2}}\right)  $
for any $\gamma\in E.$
\end{lemma}

\begin{proof}
\noindent Since $N\varphi^{f}=p^{f}\varphi^{f}N,$ we have $A\mathrm{Nm}%
_{\varphi}\left(  P\right)  =p^{f}\mathrm{Nm}_{\varphi}\left(  P\right)  A,$
hence%
\begin{align*}
\mathrm{Nm}_{\varphi}(\vec{a})\cdot\vec{a}_{11}  &  =p^{f}\cdot\mathrm{Nm}%
_{\varphi}(\vec{a})\cdot\vec{a}_{11},\ \mathrm{Nm}_{\varphi}(\vec{b})\cdot
\vec{a}_{12}=p^{f}\cdot\mathrm{Nm}_{\varphi}(\vec{a})\cdot\vec{a}%
_{12},\ \mathrm{Nm}_{\varphi}(\vec{c})\cdot\vec{a}_{13}=p^{f}\cdot
\mathrm{Nm}_{\varphi}(\vec{a})\cdot\vec{a}_{13},\\
\mathrm{Nm}_{\varphi}(\vec{a})\cdot\vec{a}_{21}  &  =p^{f}\cdot\mathrm{Nm}%
_{\varphi}(\vec{b})\cdot\vec{a}_{21},\ \mathrm{Nm}_{\varphi}(\vec{b})\cdot
\vec{a}_{22}=p^{f}\cdot\mathrm{Nm}_{\varphi}(\vec{b})\cdot\vec{a}%
_{22},\ \mathrm{Nm}_{\varphi}(\vec{c})\cdot\vec{a}_{23}=p^{f}\cdot
\mathrm{Nm}_{\varphi}(\vec{b})\cdot\vec{a}_{23},\\
\mathrm{Nm}_{\varphi}(\vec{a})\cdot\vec{a}_{31}  &  =p^{f}\cdot\mathrm{Nm}%
_{\varphi}(\vec{c})\cdot\vec{a}_{31},\ \mathrm{Nm}_{\varphi}(\vec{b})\cdot
\vec{a}_{32}=p^{f}\cdot\mathrm{Nm}_{\varphi}(\vec{c})\cdot\vec{a}%
_{32},\ \mathrm{Nm}_{\varphi}(\vec{c})\cdot\vec{a}_{33}=p^{f}\cdot
\mathrm{Nm}_{\varphi}(\vec{c})\cdot\vec{a}_{33}.
\end{align*}
Since the eigenvalues of frobenius are distinct, $\vec{a}_{11}=\vec{a}%
_{22}=\vec{a}_{33}=\vec{0}.$ Moreover, if $\vec{a}_{ij}\neq\vec{0}$ then
$\vec{a}_{ji}=\vec{0}$ for all $i,j.$ We show that at most two entries of
$[N]_{\underline{e}}$ are nonzero and give their precise formulas.

$\left(  i\right)  $ If $\vec{a}_{12}\neq\vec{0},$ then $\mathrm{Nm}_{\varphi
}(\vec{b})=p^{f}\mathrm{Nm}_{\varphi}(\vec{a})$ and the equation $\vec{b}%
\cdot\vec{a}_{12}=p\cdot\vec{a}\cdot\varphi\left(  \vec{a}_{12}\right)  $ and
Lemma \ref{norm lemma} imply that\bigskip%
\[
\vec{a}_{12}=a_{12}\left(  1,\frac{b\left(  0\right)  }{pa\left(  0\right)
},\frac{b\left(  0\right)  b\left(  1\right)  }{p^{2}a\left(  0\right)
a\left(  1\right)  },\cdots,\frac{b\left(  0\right)  b\left(  1\right)  \cdots
b\left(  f-2\right)  }{p^{f-1}a\left(  0\right)  a\left(  1\right)  \cdots
a\left(  f-2\right)  }\right)
\]
for some $a_{12}\in E^{\times}.$ If $\vec{a}_{13}\neq\vec{0},$ then
$\mathrm{Nm}_{\varphi}(\vec{c})=p^{f}\mathrm{Nm}_{\varphi}(\vec{a})$ a
contradiction since the eigenvalues of frobenius are distinct. Hence $\vec
{a}_{13}=\vec{0},$ and similarly $\vec{a}_{32}=\vec{0}\ $and $\vec{a}%
_{21}=\vec{0}.$

$\left(  ia\right)  $ If $\vec{a}_{31}\neq\vec{0},$ then $\mathrm{Nm}%
_{\varphi}(\vec{a})=p^{f}\mathrm{Nm}_{\varphi}(\vec{c})$ and the equation
$\vec{a}\cdot\vec{a}_{31}=p\cdot\vec{c}\cdot\varphi\left(  \vec{a}%
_{31}\right)  $ and Lemma \ref{norm lemma} imply that
\[
\vec{a}_{31}=a_{31}\left(  1,\frac{a\left(  0\right)  }{pc\left(  0\right)
},\frac{a\left(  0\right)  \alpha\left(  1\right)  }{p^{2}c\left(  0\right)
c\left(  1\right)  },\cdots,\frac{a\left(  0\right)  \alpha\left(  1\right)
\cdots a\left(  f-2\right)  }{p^{f-1}c\left(  0\right)  c\left(  1\right)
\cdots c\left(  f-2\right)  }\right)
\]
for some $a_{31}\in E^{\times}.$

$\left(  ib\right)  $ If $\vec{a}_{23}\neq\vec{0},$ then $\mathrm{Nm}%
_{\varphi}(\vec{c})=p^{f}\mathrm{Nm}_{\varphi}(\vec{b})$ and and the equation
$\vec{c}\cdot\vec{a}_{23}=p\cdot\vec{b}\cdot\varphi\left(  \vec{a}%
_{23}\right)  $ and Lemma \ref{norm lemma} imply that%
\[
\vec{a}_{23}=a_{23}\left(  1,\frac{c\left(  0\right)  }{pb\left(  0\right)
},\frac{c\left(  0\right)  c\left(  1\right)  }{p^{2}b\left(  0\right)
b\left(  1\right)  },\cdots,\frac{c\left(  0\right)  c\left(  1\right)  \cdots
c\left(  f-2\right)  }{p^{f-1}b\left(  0\right)  b\left(  1\right)  \cdots
b\left(  f-2\right)  }\right)
\]
for some $a_{23}\in E^{\times}.$ Notice that at most one of the $\vec{a}_{31}$
and $\vec{a}_{23}$ can be nonzero. Hence%
\[
\lbrack N]_{\underline{e}}=\left(
\begin{array}
[c]{ccc}%
\vec{0} & \vec{a}_{12} & \vec{0}\\
\vec{0} & \vec{0} & \vec{a}_{23}\\
\vec{a}_{31} & \vec{0} & \vec{0}%
\end{array}
\right)
\]
with%
\begin{align*}
\vec{a}_{12}  &  =a_{12}\left(  1,\frac{b\left(  0\right)  }{pa\left(
0\right)  },\frac{b\left(  0\right)  b\left(  1\right)  }{p^{2}a\left(
0\right)  a\left(  1\right)  },\cdots,\frac{b\left(  0\right)  b\left(
1\right)  \cdots b\left(  f-2\right)  }{p^{f-1}a\left(  0\right)  a\left(
1\right)  \cdots a\left(  f-2\right)  }\right)  \ \text{for some }a_{12}\in
E,\\
\vec{a}_{31}  &  =a_{31}\left(  1,\frac{a\left(  0\right)  }{pc\left(
0\right)  },\frac{a\left(  0\right)  \alpha\left(  1\right)  }{p^{2}c\left(
0\right)  c\left(  1\right)  },\cdots,\frac{a\left(  0\right)  \alpha\left(
1\right)  \cdots a\left(  f-2\right)  }{p^{f-1}c\left(  0\right)  c\left(
1\right)  \cdots c\left(  f-2\right)  }\right)  \ \text{for some }a_{31}\in
E,\\
\vec{a}_{23}  &  =a_{23}\left(  1,\frac{c\left(  0\right)  }{pb\left(
0\right)  },\frac{c\left(  0\right)  c\left(  1\right)  }{p^{2}b\left(
0\right)  b\left(  1\right)  },\cdots,\frac{c\left(  0\right)  c\left(
1\right)  \cdots c\left(  f-2\right)  }{p^{f-1}b\left(  0\right)  b\left(
1\right)  \cdots b\left(  f-2\right)  }\right)  \ \text{for some }a_{23}\in E.
\end{align*}
If $a_{12}\neq0$ then $\mathrm{Nm}_{\varphi}(\vec{b})=p^{f}\mathrm{Nm}%
_{\varphi}(\vec{a}),$ if $a_{31}\neq0$ then $\mathrm{Nm}_{\varphi}(\vec
{a})=p^{f}\mathrm{Nm}_{\varphi}(\vec{c})$ and if $a_{23}\neq0$ then
$\mathrm{Nm}_{\varphi}(\vec{c})=p^{f}\mathrm{Nm}_{\varphi}(\vec{b})$ and at
most two of the entries of the matrix of $[N]_{\underline{e}}$ are nonzero.
Arguing similarly for the remaining possibilities we see that $[N]_{\underline
{e}}$ has one of the following shapes:

$\left(  i\right)  \ $%
\[
\lbrack N]_{\underline{e}}=\left(
\begin{array}
[c]{ccc}%
\vec{0} & \vec{a}_{12} & \vec{0}\\
\vec{0} & \vec{0} & \vec{a}_{23}\\
\vec{a}_{31} & \vec{0} & \vec{0}%
\end{array}
\right)
\]
with%
\begin{align*}
\vec{a}_{12}  &  =a_{12}\left(  1,\frac{b\left(  0\right)  }{pa\left(
0\right)  },\frac{b\left(  0\right)  b\left(  1\right)  }{p^{2}a\left(
0\right)  a\left(  1\right)  },\cdots,\frac{b\left(  0\right)  b\left(
1\right)  \cdots b\left(  f-2\right)  }{p^{f-1}a\left(  0\right)  a\left(
1\right)  \cdots a\left(  f-2\right)  }\right)  \ \text{for some }a_{12}\in
E,\\
\vec{a}_{31}  &  =a_{31}\left(  1,\frac{a\left(  0\right)  }{pc\left(
0\right)  },\frac{a\left(  0\right)  \alpha\left(  1\right)  }{p^{2}c\left(
0\right)  c\left(  1\right)  },\cdots,\frac{a\left(  0\right)  \alpha\left(
1\right)  \cdots a\left(  f-2\right)  }{p^{f-1}c\left(  0\right)  c\left(
1\right)  \cdots c\left(  f-2\right)  }\right)  \ \text{for some }a_{31}\in
E,\\
\vec{a}_{23}  &  =a_{23}\left(  1,\frac{c\left(  0\right)  }{pb\left(
0\right)  },\frac{c\left(  0\right)  c\left(  1\right)  }{p^{2}b\left(
0\right)  b\left(  1\right)  },\cdots,\frac{c\left(  0\right)  c\left(
1\right)  \cdots c\left(  f-2\right)  }{p^{f-1}b\left(  0\right)  b\left(
1\right)  \cdots b\left(  f-2\right)  }\right)  \ \text{for some }a_{23}\in E.
\end{align*}
If $a_{12}\neq0$ then $\mathrm{Nm}_{\varphi}(\vec{b})=p^{f}\mathrm{Nm}%
_{\varphi}(\vec{a}),$ if $a_{31}\neq0$ then $\mathrm{Nm}_{\varphi}(\vec
{a})=p^{f}\mathrm{Nm}_{\varphi}(\vec{c})$ and if $a_{23}\neq0$ then
$\mathrm{Nm}_{\varphi}(\vec{c})=p^{f}\mathrm{Nm}_{\varphi}(\vec{b}).$ At most
two of the entries of the matrix of $[N]_{\underline{e}}$ are nonzero.

$\left(  ii\right)  $%
\[
\lbrack N]_{\underline{e}}=\left(
\begin{array}
[c]{ccc}%
\vec{0} & \vec{0} & \vec{a}_{13}\\
\vec{a}_{21} & \vec{0} & \vec{0}\\
\vec{0} & \vec{a}_{32} & \vec{0}%
\end{array}
\right)
\]
with%
\begin{align*}
\vec{a}_{13}  &  =a_{13}\left(  1,\frac{c\left(  0\right)  }{pa\left(
0\right)  },\frac{c\left(  0\right)  c\left(  1\right)  }{p^{2}a\left(
0\right)  a\left(  1\right)  },\cdots,\frac{c\left(  0\right)  c\left(
1\right)  \cdots c\left(  f-2\right)  }{p^{f-1}a\left(  0\right)  a\left(
1\right)  \cdots a\left(  f-2\right)  }\right)  \ \text{for some }a_{13}\in
E,\\
\vec{a}_{21}  &  =a_{21}\left(  1,\frac{a\left(  0\right)  }{pb\left(
0\right)  },\frac{a\left(  0\right)  a\left(  1\right)  }{p^{2}b\left(
0\right)  b\left(  1\right)  },\cdots,\frac{a\left(  0\right)  a\left(
1\right)  \cdots a\left(  f-2\right)  }{p^{f-1}b\left(  0\right)  b\left(
1\right)  \cdots b\left(  f-2\right)  }\right)  \ \text{for some }a_{21}\in
E,\\
\vec{a}_{32}  &  =a_{32}\left(  1,\frac{b\left(  0\right)  }{pc\left(
0\right)  },\frac{b\left(  0\right)  b\left(  1\right)  }{p^{2}c\left(
0\right)  c\left(  1\right)  },\cdots,\frac{b\left(  0\right)  b\left(
1\right)  \cdots b\left(  f-2\right)  }{p^{f-1}c\left(  0\right)  c\left(
1\right)  \cdots c\left(  f-2\right)  }\right)  \ \text{for some }a_{32}\in E.
\end{align*}
If $a_{13}\neq0$ then $\mathrm{Nm}_{\varphi}(\vec{c})=p^{f}\mathrm{Nm}%
_{\varphi}(\vec{a}),$ if $a_{21}\neq0$ then $\mathrm{Nm}_{\varphi}(\vec
{c})=p^{f}\mathrm{Nm}_{\varphi}(\vec{b})$ and if $a_{32}\neq0$ then
$\mathrm{Nm}_{\varphi}(\vec{b})=p^{f}\mathrm{Nm}_{\varphi}(\vec{c}).$ At most
two of the entries of the matrix of $[N]_{\underline{e}}$ are nonzero.

\bigskip$\left(  iii\right)  $%
\[
\lbrack N]_{\underline{e}}=\left(
\begin{array}
[c]{ccc}%
\vec{0} & \vec{a}_{12} & \vec{0}\\
\vec{0} & \vec{0} & \vec{a}_{23}\\
\vec{a}_{31} & \vec{0} & \vec{0}%
\end{array}
\right)
\]
with%
\begin{align*}
\vec{a}_{12}  &  =a_{12}\left(  1,\frac{b\left(  0\right)  }{pa\left(
0\right)  },\frac{b\left(  0\right)  b\left(  1\right)  }{p^{2}a\left(
0\right)  a\left(  1\right)  },\cdots,\frac{b\left(  0\right)  b\left(
1\right)  \cdots b\left(  f-2\right)  }{p^{f-1}a\left(  0\right)  a\left(
1\right)  \cdots a\left(  f-2\right)  }\right)  \ \text{for some }a_{12}\in
E,\\
\vec{a}_{23}  &  =a_{23}\left(  1,\frac{c\left(  0\right)  }{pb\left(
0\right)  },\frac{c\left(  0\right)  c\left(  1\right)  }{p^{2}b\left(
0\right)  b\left(  1\right)  },\cdots,\frac{c\left(  0\right)  c\left(
1\right)  \cdots c\left(  f-2\right)  }{p^{f-1}b\left(  0\right)  b\left(
1\right)  \cdots b\left(  f-2\right)  }\right)  \ \text{for some }a_{23}\in
E,\\
\vec{a}_{31}  &  =a_{31}\left(  1,\frac{a\left(  0\right)  }{pc\left(
0\right)  },\frac{a\left(  0\right)  a\left(  1\right)  }{p^{2}c\left(
0\right)  c\left(  1\right)  },\cdots,\frac{a\left(  0\right)  a\left(
1\right)  \cdots a\left(  f-2\right)  }{p^{f-1}c\left(  0\right)  c\left(
1\right)  \cdots c\left(  f-2\right)  }\right)  \ \text{for some }a_{31}\in E.
\end{align*}
If $a_{12}\neq0$ then $\mathrm{Nm}_{\varphi}(\vec{b})=p^{f}\mathrm{Nm}%
_{\varphi}(\vec{a}),$ if $a_{23}\neq0$ then $\mathrm{Nm}_{\varphi}(\vec
{c})=p^{f}\mathrm{Nm}_{\varphi}(\vec{b})$ and if $a_{31}\neq0$ then
$\mathrm{Nm}_{\varphi}(\vec{a})=p^{f}\mathrm{Nm}_{\varphi}(\vec{c}).$ At most
two of the entries of the matrix of $[N]_{\underline{e}}$ are nonzero.
\end{proof}


\begin{thebibliography}{9}                                                                                                %


\bibitem[1]{1}\bigskip G. Dousmanis, Rank two filtered $(\varphi,N)$-modules
with Galois descent data and coefficients, Transactions of the American
Mathematical Society, Vol. 362, N. 7, July 2010, pp. 3883-3910.
\end{thebibliography}
\end{document}